\documentclass[12pt,a4paper,reqno,oneside]{amsart}
\usepackage{t1enc}
\usepackage{times}
\usepackage{amssymb}
\usepackage[mathscr]{euscript}
\usepackage{graphicx}
\usepackage{hyperref}

\usepackage{comment}

\usepackage{float}
\usepackage{epstopdf}

\usepackage{tikz}

\numberwithin{equation}{section}

\newtheorem{thm}{Theorem}[section]
\newtheorem{prop}[thm]{Proposition}
\newtheorem{lem}[thm]{Lemma}
\newtheorem{cor}[thm]{Corollary}

\theoremstyle{remark}
\newtheorem{rem}[thm]{Remark}

\theoremstyle{definition}

\newcommand*{\rom}[1]{\expandafter\@slowromancap\romannumeral #1@} 
\renewcommand{\phi}{\varphi} 
\newcommand{\diff}{\mbox{d}} 
\newcommand{\R}{\mathbb{R}} 

\newcommand{\nooutput}[1]{}

\begin{document}

\title[Strong Existence and Higher Order Differentiability
of FBM driven SDE's]{Strong Existence and Higher Order Fr\'{e}chet Differentiability
of Stochastic Flows of Fractional Brownian Motion driven SDEs with Singular
Drift}

\date{\today}

\author[D. Ba\~{n}os]{David Ba\~{n}os}
\address[David Ba\~{n}os]{\\
Department of Mathematics \\
University of Oslo\\
P.O. Box 1053, Blindern\\
N--0316 Oslo, Norway}
\email[]{davidru@math.uio.no}
\author[T. Nilssen]{Torstein Nilssen}
\address{T. Nilssen: CMA, Department of Mathematics, University of Oslo, Moltke Moes vei 35, P.O. Box 1053 Blindern, 0316 Oslo, Norway. Funded by the Norwegian Research Council (Project
230448/F20) }
\email{torsteka@math.uio.no}
\author[F. Proske]{Frank Proske}
\address{F. Proske: CMA, Department of Mathematics, University of Oslo, Moltke Moes vei 35, P.O. Box 1053 Blindern, 0316 Oslo, Norway.}
\email{proske@math.uio.no} 
 
\keywords{SDEs, Compactness criterion, irregular drift, Malliavin calculus, stochastic flows, Sobolev derivative.}
\subjclass[2010]{60H10, 49N60}


\begin{abstract}
In this paper we present a new method for the construction of strong
solutions of SDE's with merely integrable drift coefficients driven by a
multidimensional fractional Brownian motion with Hurst parameter $H<\frac{1}{%
2}.$ Furthermore, we prove the rather surprising result of the higher order
Fr\'{e}chet differentiability of stochastic flows of such SDE's in the case
of a small Hurst parameter. In establishing these results we use techniques
from Malliavin calculus combined with new ideas based on a "local time
variational calculus". We expect that our general approach can be also
applied to the study of certain types of stochastic partial differential
equations as e.g. stochastic conservation laws driven by rough paths.
\end{abstract}

\maketitle

\section{Introduction}

Consider a fractional Brownian motion $B_{t}^{H}$, $t\geq 0$
with Hurst parameter $H\in (0,1)$ on a probability space $\left( \Omega ,%
\mathfrak{A},P \right) ,$ that is a centered Gaussian process with a
covariance structure $R_{H}(t,s)$ given by%
\begin{equation*}
R_{H}(t,s)= \textrm{ E}[B_{t}^{H}B_{s}^{H}]=\frac{1}{2} \left(s^{2H}+t^{2H}-\left\vert
t-s\right\vert ^{2H} \right)
\end{equation*}%
for all $t,s\geq 0$. The fractional Brownian motion, which is a Brownian
motion in the case $H=\frac{1}{2}$, enjoys the property of self-similarity,
that is 
\begin{equation*}
\{B_{\alpha t}^{H}\}_{t\geq 0}\overset{law}{=}\{\alpha
^{H}B_{t}^{H}\}_{t\geq 0}
\end{equation*}%
for all $\alpha >0$. In fact the fractional Brownian motion, which has a
version with $H-\varepsilon$-H\"{o}lder continuous paths for every $\varepsilon\in (0,H)$, is the only stationary Gaussian
process satisfying the latter property. On the other hand this process is
neither a Markov process nor a (weak) semimartingale and it is a very
irregular process in the sense of rough paths for small Hurst parameters.
See e.g. \cite{Nua10} and the references therein for more information about fractional
Brownian motion.

In this article we aim at analysing solutions $X^{x}$ of the
stochastic differential equation (SDE)%
\begin{equation}
X_{t}^{x}=x+\int_{0}^{t}b(s,X_{s}^{x})ds+B_{t}^{H}, \quad 0\leq t\leq T, \quad x\in \mathbb{R%
}^{d},  \label{VI_SDE}
\end{equation}%
where $B^{H}$ is a $d$-dimensional fractional Brownian motion, whose components are one-dimensional independent fractional Brownian motions as defined above, with Hurst parameter $H\in (0,\frac{1}{2})$ with respect to a $P$-augmented
filtration $\mathcal{F}=\{\mathcal{F}_{t}\}_{0\leq t\leq T}$ generated by $B^{H}$ and where $b:[0,T]\times \mathbb{R}^{d}\longrightarrow \mathbb{R}^{d}$ is a
Borel-measurable function.

If we impose a global Lipschitz and a linear growth condition uniformly in time on the drift
coefficient $b$ in (\ref{VI_SDE}), we can use the Picard iteration scheme to
obtain a unique global strong solution to the SDE (\ref{VI_SDE}), that is a $\mathcal{F}-$adapted solution $X_{t}^{x}$ to (\ref{VI_SDE}), which is a
measurable $L^{2}(\Omega)$-functional of the driving noise.

However, a variety of important applications of such SDE's to stochastic
control theory (in the case of $H=\frac{1}{2}$) (see \cite{Kry80}) or to
the statistical mechanics of infinite particle systems (see \cite{KR05}) show that the use of SDE's with regular coefficients in the sense of
Lipschitzianity as models for random phenomena is not suitable and that one
is forced to study such equations with coefficients which are irregular,
that is discontinuous or merely measurable.

One objective of our paper is the construction of unique strong solutions to
the SDE (\ref{VI_SDE}) driven by rough paths in the case of multidimensional
fractional noise $B^{H}$ for Hurst parameters $H<\frac{1}{2}$
and drift coefficients%
\begin{align}\label{L1inf}
b\in  L^1(\R^d; L^\infty([0,T],\R^d)) \cap L^\infty(\R^d ; L^\infty ([0,T],\R^d)).
\end{align}
In proving this new result, we employ tools from Malliavin Calculus and
local time techniques.

The analysis of strong solutions to (\ref{VI_SDE}) has been a very active field
of research in various branches of mathematics over the last decades. A
foundational result in this direction of research was first obtained by
Zvonkin in the beginning of the 1970ties \cite{Zvon74}, who showed the existence of a unique strong solution of one-dimensional Brownian motion driven SDE's (\ref{VI_SDE}), when the drift coefficient $b$ is merely bounded and measurable. A few years later on, the latter result was generalised by Veretennikov \cite{Ver79} to the multidimensional case.

More recently, Krylov and R\"{o}ckner \cite{KR05} gave the construction of unique
strong solutions to (\ref{VI_SDE}) under integrability conditions on the
(time-inhomogeneous) drift coefficient $b$. See also the articles \cite{GyM01}
or \cite{GyK96}. In this context, we shall also mention the
generalization of Zvonkin's result to the case of stochastic evolution
equations in Hilbert spaces with bounded and measurable drift coefficients
\cite{DPFPR13}, where the authors use solutions to infinite-dimensional Kolmogorov equations to recast the singular drift term of the evolution equation in terms of a more regular expression ("It\^{o}-Tanaka-Zvonkin trick").

In all of the above mentioned works the common technique of the authors for
the construction of strong solutions rests on the so-called Yamada-Watanabe
principle (see \cite{YW71}), which entails strong uniqueness of solutions to SDE's, if pathwise uniqueness of (weak) solutions holds.

In fact, in order to ensure strong uniqueness of solutions, the above
authors construct weak solutions to SDE's, which are not necessarily
Brownian functionals, by means e.g. of \cite{GyK96},\cite{GyM01}, Skorokhod embedding combined with Krylov's estimates and verify pathwise uniqueness by using solutions of parabolic partial differential equations (see e.g. \cite{Zvon74}, \cite{Ver79} or \cite{KR05}).

We remark that the techniques of these authors for proving pathwise
uniqueness are not applicable to SDE's driven by fractional Brownian motion,
since the fractional Brownian is neither a Markov process nor a
semimartingale for Hurst parameters $H\neq \frac{1}{2}$.

Further, we emphasise that our method, which is not only limited to Markov
or semimartingale solutions of SDE's, gives a direct construction of strong
solutions and provides a construction principle, which can be considered the
converse to that of Yamada-Watanabe: We prove the existence of strong
solutions and uniqueness in law to guarantee strong uniqueness.

The SDE (\ref{VI_SDE}) for fractional Brownian initial noise has been already studied by various authors in the literature:

The case $d=1$ for Hurst parameters $H\in (0,1)$ was treated in \cite{nualart.ouknine.03}, where the authors prove
strong uniqueness for linear growth drift in the case $H<\frac{1}{2}$ by invoking a method based on the comparison theorem. See also \cite{nualart.ouknine.02}.

Let us also mention the recent work of Catellier, Gubinelli \cite[2016]{CG}, which in fact came to our attention,
after the first draf (2015). In their striking paper, which extends the results of Davie \cite{Davie} to the case of a fractional Brownian noise, the authors study the problem, which fractional Brownian paths actually regularize solutions to the SDE (\ref{VI_SDE}) for $H\in (0,1)$. The (unique) solutions constructed in \cite{CG} are \emph{path by path} with respect to time-dependent vector fields $b$ in the Besov-H\"{o}lder space $B_{\infty,\infty }^{\alpha +1}$, $\alpha \in \mathbb{R}$, where in the distributional case the drift term of the SDE
is given by a non-linear Young type of integral based on an averaging operator. In proving existence and uniqueness results the authors use the Leray-Schauder-Tychonoff fixed point theorem and a comparison principle in connection with an average translation operator. Further, Lipschitz-regularity of solutions with respect to initial values under certain conditions is shown. In this context, we also refer to the PhD thesis of Catellier \cite{Catellier} , where the
author e.g. constructed (weak controlled) solutions to rough transport equations for vector fields $b$ statisfying a linear growth condition and $%
divb\in L^{\infty }([0,T]\times \mathbb{R}^{d})$ by using rough path theory.
Further, it is worth mentioning the paper of Chouk, Gubinelli \cite{Chouk1}. Here the authors analyze modulated non-linear Schr\"{o}%
dinger equations and improve well-posedness of such equations by means of the irregularity of the modulation. Their methods rest on rough path theory and an extension of Strichartz estimates to the case of Brownian modulation. See also \cite{Chouk2} in connection with the Korteweg-de Vries equation.

Finally, we refer to other recent works by Hu, Khoa, Mytnik \cite{Mytnik1}, which pertains to the study of the Brox
diffusion, and Butkovski, Mytnik \cite{Mytnik2}, where the authors obtain results on the regularization by (space
time white) noise of solutions to a non-Lipschitz stochastic heat equation and the associated flow. Moreover, path by path unique solutions in the sense of Davie \cite{Davie} are shown.

The techniques used in our paper are based on Malliavin calculus and are very different from those in the above mentioned papers- in spite of some (first impression) similarities regarding our estimates in Prop. 3.3 and 3.4 to the article of Davie \cite{Davie}, which is limited to the case of Brownian motion and whose approach doesn't carry over to our situation.  Further, the existence and uniqueness results for strong solutions to (\ref{VI_SDE}) for all (multidimensional)
vector fields $b$ as in \eqref{L1inf} established in this paper are not covered by the work \cite{CG}. Moreover, our method- and this
is a characteristic feature of our article- allows for the proof of higher order differentiability of stochastic flows associated with such solutions, provided the Hurst parameter is small enough.

\bigskip

Another crucial objective of our article is the study of the regularity of
stochastic flows of the SDE (\ref{VI_SDE}), that is the regularity of 
\begin{equation*}
(x\longmapsto X_{t}^{x})
\end{equation*}%
in the initial condition $x\in \mathbb{R}^{d},$ when the vector field $b$ is
discontinuous.

The motivation for this study comes from the deterministic case:%
\begin{equation}
\frac{d}{dt}X_{t}^{x}=u(t,X_{t}^{x}), \quad t\geq 0, \quad X_{0}^{x}=x,  \label{VI_Classical}
\end{equation}%
where $u:[0,\infty )\times \mathbb{R}^{d}\longrightarrow \mathbb{R}^{d}$ is
a vector field. Here, the solution $X:[0,\infty )\times \mathbb{R}%
^{d}\longrightarrow \mathbb{R}^{d}$ to (\ref{VI_Classical}) may e.g. stand for
the flow of fluid particles with respect to the velocity field of an
incompressible inviscid fluid whose dynamics is described by an
incompressible Euler equation%
\begin{equation}
u_{t}+(Du)u+\triangledown P=0,\quad \nabla \cdot u=0,  \label{VI_Euler}
\end{equation}%
where $P:[0,\infty )\times \mathbb{R}^{d}\longrightarrow \mathbb{R}$ is the
pressure field.

Solutions of (\ref{VI_Euler}) may be singular. Therefore a better understanding of the regularity of solutions of equation (\ref{VI_Euler}) requires the study of flows of ODE's (\ref{VI_Classical}) driven by irregular vector fields.

If $u$ is Lipschitz continuous it is well-known that the unique flow $%
X:[0,\infty )\times \mathbb{R}^{d}\longrightarrow \mathbb{R}^{d}$ in (\ref{VI_Classical}) is Lipschitzian. The latter classical result was generalized by Di Perna and Lions in their celebrated paper \cite{DPL89} to the case $u\in L^1([0,T]; W_{loc}^{1,p})$ and $\nabla\cdot u\in L^1([0,T];L^{\infty})$, for which the authors construct a unique generalized flow $X$
to (\ref{VI_Classical}). Later on the latter result was extended by Ambrosio \cite{ambrosio.04} to the case of vector fields of bounded
variation.

However, it turns out that the superposition of the ODE (\ref{VI_Classical}) by
a Brownian noise $B$, that is
\begin{equation}
dX_{t}=u(t,X_{t})dt+dB_{t}, \quad s,t\geq 0, \quad X_{s}=x\in \mathbb{R}^{d}
\label{VI_SDEBrownian}
\end{equation}%
has a strong regularising effect on its flow $\mathbb{R}^{d}\ni x\longmapsto
\phi _{s,t}(x)\in \mathbb{R}^{d}.$ 

Using techniques similar to those in this paper, but without arguments based
on local time, it was shown in Mohammed, Nilssen, Proske \cite{MNP14} for merely \emph{bounded measurable} drift coefficients $u$ that $\phi _{s,t}$ is a stochastic flow of Sobolev diffeomorphisms with  
\begin{equation*}
\phi _{s,t}(\cdot ), \phi _{s,t}^{-1}(\cdot )\in L^{2}(\Omega
,W^{1,p}(\mathbb{R}^{d};w))
\end{equation*}%
for all $s,t$ and $p\in (1,\infty )$, where $W^{1,p}(\mathbb{R}^{d};w)$ is a
weighted Sobolev space with weight function $w:\mathbb{R}^{d}\longrightarrow
\lbrack 0,\infty )$.

As an application of this result the authors constructed Sobolev
differentiable unique (weak) solutions of the (Stratonovich) stochastic
transport equation with multiplicative noise of the form%
\begin{equation*}
\left\{ 
\begin{array}{l}
d_{t}v(t,x)+(u(t,x)\cdot Dv(t,x))dt+\sum_{i=1}^{d}e_{i}\cdot Dv(t,x)\circ
dB_{t}^{i}=0 \\ 
u(0,x)=u_{0}(x),%
\end{array}%
\right. 
\end{equation*}%
where $u$ is bounded and measurable, $u_{0}\in C_{b}^{1}$ and where $\left\{
e_{i}\right\} _{i=1}^{d}$ is a basis of $\mathbb{R}^{d}$.

By adopting ideas in Mohammed et al. \cite{MNP14}, we mention that the latter result on the existence of stochastic flows of Sobolev diffeomorphisms was extended in \cite{Rez.14} to the case of globally integrable $u\in L^{r,q}$ for $r/d+2/q<1$ ($r$ for the spatial variable and $q$ for the temporal variable) and applied to the study of the regularity of solutions to Navier-Stokes-equations. Compare also to \cite{FedFlan.13}, where the authors employ techniques based on solutions of backward Kolmogorov equations.

\bigskip 

If the Brownian motion in (\ref{VI_SDEBrownian}) is replaced by a rougher noise
given by $B^{H}$ for $H<\frac{1}{2}$, we find in this paper for $u\in L^1(\R^d; L^\infty([0,T],\R^d)) \cap L^\infty(\R^d ; L^\infty ([0,T],\R^d))$ the rather surprising result which generalises the classical result of Kunita \cite{Kunita} for smooth coefficients, that the
stochastic flow $X:[0,\infty )\times \mathbb{R}^{d}\longrightarrow \mathbb{R}%
^{d}$ is higher order Fr\'{e}chet differentiable in the spatial variable,
that is%
\begin{equation*}
(x\longmapsto X_{t}^{x}(\omega ))\in C^{k}(\mathbb{R}^{d})
\end{equation*}%
a.s. for all $t$ and for $k\geq 1$, provided $H=H(k)$ is small enough.

In view of the above discussion in the case of Brownian noise driven
stochastic flows, the latter result raises the fundamental question whether
rough noise in the sense of $B_{\cdot }^{H}$ or a related noise with very
irregular path behaviour may considerably regularise solutions of PDE's as
e.g. transport equations, conservation laws or even Navier-Stokes equations
by perturbation. We are confident that there is an affirmative answer for a
class of interesting PDE's.

\bigskip 

Finally, we comment on that the method for the construction of higher order
Fr\'{e}chet differentiable stochastic flows of (\ref{VI_SDE}), which is- as
mentioned above- different from common techniques based on Markov processes
and semimartingales, is inspired by the works \cite{MBP10}, \cite{MMNPZ10},
\cite{MNP14}, \cite{HaaPros.14} in the case of (\ref{VI_SDE}) with
initial L\'{e}vy noise and \cite{FNP.13}, \cite{nilssen.14} in the case of stochastic partial differential equations.

More precisely, in order to construct strong solutions to (\ref{VI_SDE}) we
apply a compactness criterion for square integrable Brownian functionals from
\cite{DPMN92} to solutions $X_{t}^{n}$ of%
\begin{equation*}
dX_{t}^{n}=b_{n}(t,X_{t}^{n})dt+dB_{t}^{H},
\end{equation*}%
where $b_{n},n\geq 0$ are smooth coefficients converging to $b$ in $L^1(\R^d; L^\infty([0,T],\R^d)) \cap L^\infty(\R^d ; L^\infty ([0,T],\R^d))$ and show that $X_{t}^{n}$ converges to a solution $X_{t}$ of (\ref{VI_SDE}) in $L^{2}(\Omega)$ for all $t$. 

If, for a moment, we assume that $b$ is time-homogeneous, then in proving the existence and the higher order Fr\'{e}chet differentiability of the corresponding stochastic flow we make use of a "local time variational calculus" argument of the form%
\begin{align}\label{localintro}
\int_{\Delta _{\theta ,t}^{m}}\varkappa (s)D^{\alpha }f(B_{s}^{H})ds=\int_{(\mathbb{R}^{d})^m}D^{\alpha }f(z)L_{\varkappa }(t,z)dz=(-1)^{\left\vert \alpha
\right\vert }\int_{(\mathbb{R}^{d})^m}f(z)D^{\alpha }L_{\varkappa }(t,z)dz,
\end{align}
for $B_s^H=(B_{s_1}^H,\dots, B_{s_m}^H)$ and smooth functions $f:(\mathbb{R}^{d})^m\longrightarrow \mathbb{R}$, where $%
L_{\varkappa }(t,z)$ is a spatially differentiable local time on the simplex $%
\Delta _{\theta ,t}^{m}=\{(s_{1},...,s_{m})\in \lbrack 0,T]^{m}:\theta
<s_{1}<...<s_{m}<t\}$, scaled by a function $\varkappa(s_1,\dots,s_m)$ ($D^{\alpha }$ is the partial
derivative of order $\left\vert \alpha \right\vert $). Actually, we generalise the above argument to time dependent smooth functions $f:[0,T]^m\times (\R^{d})^m\rightarrow \R$ and hence the intuition of the above "local time" argument is somehow not tangible any longer. In other words, we show that there exists a well-defined object $\Lambda_\alpha^f (\theta,t,z)$ in $L^2(\Omega)$ the size of which can be estimated by means of a norm of $f$ and not by its derivative such that the following integration by parts formula holds true
\begin{align}\label{localintro2}
\int_{\Delta _{\theta ,t}^{m}}D^{\alpha }f(s,B_{s}^{H})ds=\int_{(\mathbb{R}^{d})^m} \Lambda_\alpha^f (\theta,t,z) dz, \quad P-a.s.
\end{align}
where the above formula coincides with \eqref{localintro} for time-homogeneous functions.

\bigskip 

We expect that our approach can be also applied to the study of solutions of
the following stochastic equations:

\bigskip 

\begin{equation*}
dX_{t}=(AX_{t}+b(X_{t}))dt+Qd\textbf{W}_{t},
\end{equation*}%
for (mild) solutions $X_{t},$ where $A$ is a densely defined linear operator
(of parabolic type) on a separable Hilbert space $H$, $b:H\longrightarrow H$
is an irregular function, $Q$ a Hilbert-Schmidt operator and $\textbf{W}$ a (non-H\"{o}lder continuous) "cylindrical" Gaussian noise.

\bigskip 

On the other hand, using our method we may also examine equations of the type
\begin{equation*}
dX_{t}=dA_{t}+dB_{t}^{H},
\end{equation*}%
where $A_{t}$ is a process of bounded variation which arises from limits of
the form
$$\displaystyle \lim_{n\to \infty }\int_{0}^{t}b_{n}(X_{s})ds$$
for coefficients $b_{n},n\geq 0$. See \cite{BC.03} in the Brownian case and the works \cite{BLPP}, \cite{ABP}.

\bigskip 

Our paper is organised as follows: In Section 2 we introduce the
mathematical framework of the article and define in Section 3 the random field $\Lambda_{\alpha}^f$ of \eqref{localintro2}, which we show to be high-order differentiable in the spatial variable for small Hurst
parameters. In Section 4 we establish the existence of a unique strong
solution to the SDE (\ref{VI_SDE}) under integrability conditions on the drift
coefficient $b$. Section 5 is devoted to the study of the regularity
properties of stochastic flows of (\ref{VI_SDE}).   

%

\section{Framework}\label{VI_Frame}
In this section we recollect some specifics on fractional calculus, fractional Brownian noise and occupation measures which will be extensively used throughout the article. The reader might consult \cite{Mall97}, \cite{Mall78} or \cite{DOP08} for a general theory on Malliavin calculus for Brownian motion and \cite[Chapter 5]{Nua10} for fractional Brownian motion. Whereas for occupation measures one may review \cite{geman.horo.80} or \cite{Kar98}. We present the results in one dimension for simplicity inasmuch as we will treat the multidimensional case componentwise.

\subsection{Fractional calculus}\label{VI_fraccal}
We establish here some basic definitions and properties on fractional calculus. A general theory on this subject may be found in \cite{samko.et.al.93} and \cite{lizorkin.01}.

Let $a,b\in \R$ with $a<b$. Let $f\in L^p([a,b])$ with $p\geq 1$ and $\alpha>0$. Define the \emph{left-} and \emph{right-sided Riemann-Liouville fractional integrals} by
$$I_{a^+}^\alpha f(x) = \frac{1}{\Gamma (\alpha)} \int_a^x (x-y)^{\alpha-1}f(y)dy$$
and
$$I_{b^-}^\alpha f(x) = \frac{1}{\Gamma (\alpha)} \int_x^b (y-x)^{\alpha-1}f(y)dy$$
for almost all $x\in [a,b]$ where $\Gamma$ is the Gamma function.

Moreover, for a given integer $p\geq 1$, let $I_{a^+}^{\alpha} (L^p)$ (resp. $I_{b^-}^{\alpha} (L^p)$) denote the image of $L^p([a,b])$ by the operator $I_{a^+}^\alpha$ (resp. $I_{b^-}^\alpha$). If $f\in I_{a^+}^{\alpha} (L^p)$ (resp. $f\in I_{b^-}^{\alpha} (L^p)$) and $0<\alpha<1$ then define the \emph{left-} and \emph{right-sided Riemann-Liouville fractional derivatives} by
$$D_{a^+}^{\alpha} f(x)= \frac{1}{\Gamma (1-\alpha)} \frac{\diff}{\diff x} \int_a^x \frac{f(y)}{(x-y)^{\alpha}}dy$$
and
$$D_{b^-}^{\alpha} f(x)= \frac{1}{\Gamma (1-\alpha)} \frac{\diff}{\diff x} \int_x^b \frac{f(y)}{(y-x)^{\alpha}}dy.$$

The left- and right-sided derivatives of $f$ defined above have the following representations
$$D_{a^+}^{\alpha} f(x)= \frac{1}{\Gamma (1-\alpha)} \left(\frac{f(x)}{(x-a)^\alpha}+\alpha\int_a^x \frac{f(x)-f(y)}{(x-y)^{\alpha+1}}dy\right)$$
and
$$D_{b^-}^{\alpha} f(x)= \frac{1}{\Gamma (1-\alpha)} \left(\frac{f(x)}{(b-x)^\alpha}+\alpha\int_x^b \frac{f(x)-f(y)}{(y-x)^{\alpha+1}}dy\right).$$

Finally, observe that by construction, the following formulas hold
$$I_{a^+}^\alpha (D_{a^+}^{\alpha} f) = f$$
for all $f\in I_{a^+}^{\alpha} (L^p)$ and
$$D_{a^+}^{\alpha}(I_{a^+}^\alpha  f) = f$$
for all $f\in L^p([a,b])$ and similarly for $I_{b^-}^{\alpha}$ and $D_{b^-}^{\alpha}$.

\subsection{Shuffles}\label{VI_shuffles}
Let $m$ and $n$ be integers. We define $S(m,n)$ as the set of \emph{shuffle permutations}, i.e. the set of permutations $\sigma: \{1, \dots, m+n\} \rightarrow \{1, \dots, m+n\}$ such that $\sigma(1) < \dots < \sigma(m)$ and $\sigma(m+1) < \dots < \sigma(m+n)$.

We define the $m$-dimensional simplex for $0\leq \theta < t \leq T$,
$$
\Delta_{\theta,t}^m := \{(s_m,\dots,s_1)\in [0,T]^m : \, \theta<s_m<\cdots < s_1<t\}.
$$
The product of two simplices can be written as the following union
$$
\Delta_{\theta,t}^m \times \Delta_{\theta,t}^n = \mbox{\footnotesize $\bigcup_{\sigma \in S(m,n)} \{(w_{m+n},\dots,w_1)\in [0,T]^{m+n} : \, \theta< w_{\sigma(m+n)} <\cdots < w_{\sigma(1)} <t\} \cup \mathcal{N}$ \normalsize},
$$
where the set $\mathcal{N}$ has null Lebesgue measure. In this way, if $f_i:[0,T] \rightarrow \R$, $i=1,\dots,m+n$ are integrable functions we have
\begin{align} \label{shuffleIntegral}
\int_{\Delta_{\theta,t}^m} \prod_{j=1}^m f_j(s_j) ds_m \dots ds_1 & \int_{\Delta_{\theta,t}^n} \prod_{j=m+1}^{m+n} f_j(s_j) ds_{m+n} \dots ds_{m+1}  \notag \\
&= \sum_{\sigma\in S(m,n)} \int_{\Delta_{\theta,t}^{m+n}} \prod_{j=1}^{m+n} f_{\sigma(j)} (w_j) dw_{m+n}\cdots dw_1. 
\end{align}

We can generalize the above technical lemma, the use of which shall be clear in Section \ref{Flow}. The reader may skip this lemma and proof until Section \ref{Flow}.

\begin{lem}\label{partialshuffle}
Let $n,p$ and $k$ be non-negative integers, $k \leq n$. Assume we have integrable functions $f_j : [0,T] \rightarrow \R$, $j = 1, \dots, n$ and $g_i : [0,T] \rightarrow \R$, $i = 1, \dots, p$. We may then write
\begin{align*}
\int_{\Delta_{\theta,t}^n} f_1(s_1) \dots f_k(s_k) \int_{\Delta_{\theta, s_k}^p} g_1(r_1) \dots g_p(r_p) dr_p \dots dr_1 f_{k+1}(s_{k+1}) \dots f_n(s_n) ds_n \dots ds_1 \\
= \sum_{\sigma \in A_{n,p}} \int_{\Delta_{\theta,t}^{n+p}} h^{\sigma}_1(w_1) \dots h^{\sigma}_{n+p}(w_{n+p}) dw_{n+p} \dots dw_1,
\end{align*}
where $h^{\sigma}_l \in \{ f_j, g_i : 1 \leq j \leq n, 1 \leq i \leq p\}$. Above $A_{n,p}$ denotes a subset of permutations of $\{1, \dots, n+p\}$ such that $\# A_{n,p} \leq C^{n+p}$ for an appropriate constant $C \geq 1$, and we have defined $s_0 = \theta$.
\end{lem}

\begin{proof}
The result is proved by induction on $n$. For $n=1$ and $k=0$ the result is trivial. For $k=1$ we have
\begin{align*}
\int_{\theta}^t f_1(s_1) \int_{\Delta_{\theta,s_1}^p} g_1(r_1) \dots g_p(r_p) & dr_p \dots dr_1 ds_1 \\
 &  = \int_{\Delta_{\theta,t}^{p+1}} f_1(w_1) g_1(w_2) \dots g_p(w_{p+1})  dw_{p+1} \dots dw_1,
\end{align*}
where we have put $w_1 =s_1, w_2 =  r_1, \dots, w_{p+1} = r_p$.

Assume the result holds for $n$ and let us show that this implies that the result is true for $n+1$. Either $k=0,1$ or $2 \leq k \leq n+1$. For $k=0$ the result is trivial. For $k=1$ we have
\begin{align*}
\int_{\Delta_{\theta,t}^{n+1}} & f_1(s_1) \int_{\Delta_{\theta,s_1}^p} g_1(r_1) \dots g_p(r_p) dr_p \dots dr_1 f_2(s_2) \dots  f_{n+1}(s_{n+1}) ds_{n+1} \dots ds_1 \\
&= \int_{\theta}^t f_1(s_1) \left( \int_{\Delta_{\theta,s_1}^n} \int_{\Delta_{\theta,s_1}^p} g_1(r_1) \dots g_p(r_p) dr_p \dots dr_1 f_2(s_2) \dots  f_{n+1}(s_{n+1}) ds_{n+1} \dots ds_2 \right) ds_1 .
\end{align*}
The result follows from (\ref{shuffleIntegral}) coupled with $ \# S(n,p) = \frac{(n+p)!}{n! p!} \leq C^{n+p} \leq C^{(n+1) + p}$. For $k \geq 2$ we have from the induction hypothesis
\begin{align*}
\int_{\Delta_{\theta,t}^{n+1}}   f_1(s_1) \dots f_k(s_k) \int_{\Delta_{\theta, s_k}^p}  g_1(r_1) \dots g_p(r_p) & dr_p \dots dr_1 f_{k+1}(s_{k+1}) \dots f_{n+1}(s_{n+1}) ds_{n+1} \dots ds_1 \\
 = \int_{\theta}^t f_1(s_1)   \int_{\Delta_{\theta,s_1}^n} f_2(s_2) \dots f_k(s_k) & \int_{\Delta_{\theta, s_k}^p} g_1(r_1) \dots g_p(r_p) dr_p \dots dr_1  \\
&  \times  f_{k+1}(s_{k+1}) \dots f_{n+1}(s_{n+1}) ds_{n+1} \dots ds_2  ds_1 \\
 =   \sum_{\sigma \in A_{n,p}} \int_{\theta}^t f_1(s_1)   \int_{\Delta_{\theta,s_1}^{n+p}} & h^{\sigma}_1(w_1) \dots h^{\sigma}_{n+p}(w_{n+p}) dw_{n+p} \dots dw_1 ds_1\\
= \sum_{\tilde{\sigma} \in A_{n+1,p}} \int_{\Delta_{\theta,t}^{n+1+p}} & h^{\tilde{\sigma}}_1(w_1) \dots \tilde{h}^{\tilde{\sigma}}_{w_{n+1+p}} dw_1 \dots dw_{n+1+p},
\end{align*}
where $A_{n+1,p}$ is the set of permutations $\tilde{\sigma}$ of $\{1, \dots, n+1+p\}$ such that $\tilde{\sigma}(1) = 1$ and $\tilde{\sigma}(j+1) = \sigma(j)$, $j=1, \dots, n+p$ for some $\sigma \in A_{n,p}$ .

\end{proof}

\begin{rem}
Notice that the set $A_{n,p}$ in the above lemma also depends on $k$ but we shall not need this fact.
\end{rem}

\subsection{Fractional Brownian motion}
Let $B^H = \{B_t^H, t\in [0,T]\}$ be a $d$-dimensional \emph{fractional Brownian motion} with Hurst parameter $H\in (0,1/2)$ defined on a probability space $(\Omega,\mathfrak{A},P)$. In other words, $B^H$ is a centered Gaussian process with covariance structure
$$(R_H(t,s))_{i,j}:=\textrm{ E}[B_t^{H,(i)} B_s^{H,(j)}]=\frac{1}{2}\left(t^{2H} + s^{2H} - |t-s|^{2H} \right), \quad i,j=1,\dots,d.$$
Observe that $\textrm{ E}[|B_t^H - B_s^H|^2]= d|t-s|^{2H}$ and hence $B^H$ has stationary increments and H\"{o}lder continuous trajectories of index $H-\varepsilon$ for all $\varepsilon\in(0,H)$. Observe moreover that the increments of $B^H$, $H\in (0,1/2)$ are not independent. This fact makes computations more difficult. Another difficulty one encounters is that $B^H$ is not a semimartingale, see e.g. \cite[Proposition 5.1.1]{Nua10}.

Now we give a brief survey on how to construct fractional Brownian motion via an isometry. Since the construction can be done componentwise we present here for simplicity the one-dimensional case. Further details can be found in \cite{Nua10}.

Denote by $\mathcal{E}$ the set of step functions on $[0,T]$ and denote by $\mathcal{H}$ the Hilbert space defined as the closure of $\mathcal{E}$ with respect to the inner product 
$$\langle 1_{[0,t]} , 1_{[0,s]}\rangle_{\mathcal{H}} = R_H(t,s).$$
The mapping $1_{[0,t]} \mapsto B_t$ can be extended to an isometry between $\mathcal{H}$ and the Gaussian subspace of $L^2(\Omega)$ associated with $B^H$. Denote such isometry by $\varphi \mapsto B^H(\varphi)$. We recall the following result (see \cite[Proposition 5.1.3]{Nua10} ) which gives an integral representation of $R_H(t,s)$ when $H<1/2$:

\begin{prop}
Let $H<1/2$. The kernel
$$K_H(t,s)= c_H \left[\left( \frac{t}{s}\right)^{H- \frac{1}{2}} (t-s)^{H- \frac{1}{2}} + \left( \frac{1}{2}-H\right) s^{\frac{1}{2}-H} \int_s^t u^{H-\frac{3}{2}} (u-s)^{H-\frac{1}{2}} du\right],$$
where $c_H = \sqrt{\frac{2H}{(1-2H) \beta(1-2H , H+1/2)}}$ being $\beta$ the Beta function, satisfies
\begin{align}\label{VI_RH}
R_H(t,s) = \int_0^{t\wedge s} K_H(t,u)K_H(s,u)du.
\end{align}
\end{prop}

The kernel $K_H$ can also be represented by means of fractional derivatives as follows
$$K_H(t,s) = c_H \Gamma \left( H+\frac{1}{2}\right) s^{\frac{1}{2}-H} \left( D_{t^-}^{\frac{1}{2}-H} u^{H-\frac{1}{2}}\right)(s).$$

Consider the linear operator $K_H^{\ast}: \mathcal{E} \rightarrow L^2([0,T])$ defined by
$$(K_H^{\ast} \varphi)(s) = K_H(T,s)\varphi(s) + \int_s^T (\varphi(t)-\varphi(s)) \frac{\partial K_H}{\partial t}(t,s)dt$$
for every $\varphi \in \mathcal{E}$. Observe that $(K_H^{\ast} 1_{[0,t]})(s) = K_H(t,s)1_{[0,t]}(s)$, then from this fact and \eqref{VI_RH} we see that $K_H^{\ast}$ is an isometry between $\mathcal{E}$ and $L^2([0,T])$ which can be extended to the Hilbert space $\mathcal{H}$.

For a given $\varphi\in \mathcal{H}$ one can show the following two representations for $K_H^{\ast}$ in terms of fractional derivatives
$$(K_H^{\ast} \varphi)(s) = c_H \Gamma\left( H+\frac{1}{2}\right) s^{\frac{1}{2}-H} \left(D_{T^-}^{\frac{1}{2}-H} u^{H-\frac{1}{2}}\varphi(u)\right) (s)$$
and
\begin{align*}
(K_H^{\ast} \varphi)(s) =& \, c_H \Gamma\left( H+\frac{1}{2}\right)\left(D_{T^-}^{\frac{1}{2}-H} \varphi(s)\right) (s)\\
&+ c_H \left( \frac{1}{2}-H\right)\int_s^T  \varphi(t) (t-s)^{H-\frac{3}{2}} \left(1- \left(\frac{t}{s}\right)^{H-\frac{1}{2}}\right)dt.
\end{align*}

One can show that $\mathcal{H} = I_{T^-}^{\frac{1}{2}-H}(L^2)$ (see \cite{decreu.ustunel.98} and \cite[Proposition 6]{alos.mazet.nualart.01}).

Given the fact that $K_H^{\ast}$ is an isometry from $\mathcal{H}$ into $L^2([0,T])$ the $d$-dimensional process $W=\{W_t, t\in [0,T]\}$ defined by
\begin{align}\label{VI_WBH}
W_t := B^H((K_H^{\ast})^{-1}(1_{[0,t]}))
\end{align}
is a Wiener process and the process $B^H$ has the following representation
\begin{align}\label{VI_BHW}
B_t^H = \int_0^t K_H(t,s) dW_s,
\end{align}
see \cite{alos.mazet.nualart.01}.

Henceforward, we will denote by $W$ a standard Wiener process on a given probability space $(\Omega, \mathfrak{A}, P)$ equipped with the natural filtration $\mathcal{F}=\{\mathcal{F}_t\}_{t\in [0,T]}$ generated by $W$ augmented by all $P$-null sets and $B:=B^H$ the fractional Brownian motion with Hurst parameter $H\in (0,1/2)$ given by the representation \eqref{VI_BHW}.

Next, we give a version of Girsanov's theorem for fractional Brownian motion which is due to \cite[Theorem 4.9]{decreu.ustunel.98}. Here we present the version given in \cite[Theorem 3.1]{nualart.ouknine.02} but first we need to define an isomorphism $K_H$ from $L^2([0,T])$ onto $I_{0+}^{H+\frac{1}{2}}(L^2)$ associated with the kernel $K_H(t,s)$ in terms of the fractional integrals as follows, see \cite[Theorem 2.1]{decreu.ustunel.98}
$$(K_H \varphi)(s) = I_{0^+}^{2H} s^{\frac{1}{2}-H} I_{0^+}^{\frac{1}{2}-H}s^{H-\frac{1}{2}}  \varphi, \quad \varphi \in L^2([0,T]).$$

From this and the properties of the Riemann-Liouville fractional integrals and derivatives the inverse of $K_H$ is given by
$$(K_H^{-1} \varphi)(s) = s^{\frac{1}{2}-H} D_{0^+}^{\frac{1}{2}-H} s^{H-\frac{1}{2}} D_{0^+}^{2H} \varphi(s), \quad \varphi \in I_{0+}^{H+\frac{1}{2}}(L^2).$$

It follows that if $\varphi$ is absolutely continuous, see \cite{nualart.ouknine.02}, one can show that
\begin{align}\label{VI_inverseKH}
(K_H^{-1} \varphi)(s) = s^{H-\frac{1}{2}} I_{0^+}^{\frac{1}{2}-H} s^{\frac{1}{2}-H}\varphi'(s).
\end{align}

\begin{thm}[Girsanov's theorem for fBm]\label{VI_girsanov}
Let $u=\{u_t, t\in [0,T]\}$ be an $\mathcal{F}$-adapted process with integrable trajectories  and set
$\widetilde{B}_t^H = B_t^H + \int_0^t u_s ds, \quad t\in [0,T].$
Assume that
\begin{itemize}
\item[(i)] $\int_0^{\cdot} u_s ds \in I_{0+}^{H+\frac{1}{2}} (L^2 ([0,T])$, $P$-a.s.

\item[(ii)] $\textrm{ E}[\xi_T]=1$ where
$$\xi_T := \exp\left\{-\int_0^T K_H^{-1}\left( \int_0^{\cdot} u_r dr\right)(s)dW_s - \frac{1}{2} \int_0^T K_H^{-1} \left( \int_0^{\cdot} u_r dr \right)^2(s)ds \right\}.$$
\end{itemize}
Then the shifted process $\widetilde{B}^H$ is an $\mathcal{F}$-fractional Brownian motion with Hurst parameter $H$ under the new probability $\widetilde{P}$ defined by $\frac{d\widetilde{P}}{dP}=\xi_T$.
\end{thm}

\begin{rem}
For the multidimensional case, define
$$(K_H \varphi)(s):= ( (K_H \varphi^{(1)} )(s), \dots, (K_H  \varphi^{(d)})(s))^{\ast}, \quad \varphi \in L^2([0,T];\R^d),$$
where $\ast$ denotes transposition. Similarly for $K_H^{-1}$ and $K_H^{\ast}$.
\end{rem}

Finally, we want to use a crucial property of the fractional Brownian in this paper, which is referred to in the literature as \emph{strong (two-sided) local non-determinism} (see e.g.\cite{pitt.78} or \cite{xiao.11}). This property will essentially help us to overcome the limitations of not having independent increments of the underlying noise:
There exists a constant $K>0$, depending only on $H$ and $T$, such
that for any $t\in\left[0,T\right],0<r<t$ and for $i=1,\ldots,d,$
\begin{align}\label{eq:SLND}
\mathrm{Var}\left[B_{t}^{H,i}|\left\{ B_{s}^{H,i}:\left|t-s\right|\geq r\right\} \right]\geq Kr^{2H}.
\end{align}


\section{An integration by parts formula}\label{VI_Section3}

Let $m$ be an integer and consider a $f :[0,T]^m \times (\R^d)^m \rightarrow \R$ of the form
\begin{align}\label{f}
f(s,z)= \prod_{j=1}^m f_j(s_j,z_j),\quad s = (s_1, \dots,s_m) \in [0,T]^m, \quad z = (z_1, \dots, z_m) \in (\R^d)^m,
\end{align}
where $f_j:[0,T]\times \R^d \rightarrow \R$, $j=1,\dots,m$ are smooth functions with compact support. Moreover, consider an integrable $\varkappa:[0,T]^m\rightarrow \R$ of the form
\begin{align}\label{kappa}
\varkappa(s)= \prod_{j=1}^m \varkappa_j(s_j), \quad s\in [0,T]^m,
\end{align}
where $\varkappa_j : [0,T] \rightarrow \R$, $j=1,\dots, m$ are integrable functions.

Denote by $\alpha_j$ a multi-index and $D^{\alpha_j}$ its corresponding differential operator. For $\alpha = (\alpha_1, \dots, \alpha_m)$ considered as an element of $\mathbb{N}_0^{d\times m}$ so that $|\alpha|:= \sum_{j=1}^m \sum_{l=1}^d \alpha_{j}^{(l)}$, we write
$$
D^{\alpha}f(s,z) = \prod_{j=1}^m D^{\alpha_j} f_j(s_j,z_j).
$$

The aim of this section is to derive an integration by parts formula of the form
\begin{equation} \label{ibp}
\int_{\Delta_{\theta,t}^m} D^{\alpha}f(s,B_s) ds = \int_{(\R^d)^m} \Lambda^{f}_{\alpha} (\theta,t,z)dz ,
\end{equation}
where $B:=B^H$, for a suitable random field $\Lambda^f_{\alpha}$. In fact, we have
\begin{equation} \label{LambdaDef}
\Lambda^f_{\alpha}(\theta,t ,z) = (2 \pi)^{-dm} \int_{(\R^d)^m} \int_{\Delta_{\theta,t}^m} \prod_{j=1}^m f_j(s_j,z_j) (-i u_j)^{\alpha_j} \exp \{ -i \langle u_j, B_{s_j} - z_j \rangle\}ds du .
\end{equation}

We start by \emph{defining} $\Lambda^f_{\alpha}(\theta,t,z)$ as above and show that it is a well-defined element of $L^2(\Omega)$.

Introduce the following notation: given $(s,z) = (s_1, \dots, s_m ,z_1 \dots, z_m)  \in [0,T]^m \times (\R^d)^m$ and a shuffle $\sigma \in S(m,m)$ we write
$$
f_{\sigma}(s,z) := \prod_{j=1}^{2m} f_{[\sigma(j)]}(s_j, z_{[\sigma(j)]})
$$
and
$$\varkappa_{\sigma} (s) := \prod_{j=1}^{2m} \varkappa_{[\sigma(j)]}(s_j),$$
where $[j]$ is equal to $j$ if $1\leq j\leq m$ and $j-m$ if $%
m+1\leq j\leq 2m$.

For a multiindex $\alpha$ we define

\begin{eqnarray*}
&&\Psi _{\alpha}^{f}(\theta ,t,z) \\
&:&=\prod_{l=1}^{d}\sqrt{(2\left\vert \alpha ^{(l)}\right\vert )!}%
\sum_{\sigma \in S(m,m)}\int_{\Delta _{0,t}^{2m}}\left\vert f_{\sigma
}(s,z)\right\vert \prod_{j=1}^{2m}\frac{1}{\left\vert
s_{j}-s_{j-1}\right\vert ^{H(d+2\sum_{l=1}^{d}\alpha _{\lbrack \sigma
(j)]}^{(l)})}}ds_{1}...ds_{2m}
\end{eqnarray*}

respectively,
\begin{eqnarray*}
&&\Psi _{\alpha}^{\varkappa }(\theta ,t) \\
&:&=\prod_{l=1}^{d}\sqrt{(2\left\vert \alpha ^{(l)}\right\vert )!}%
\sum_{\sigma \in S(m,m)}\int_{\Delta _{0,t}^{2m}}\left\vert \varkappa
_{\sigma }(s)\right\vert \prod_{j=1}^{2m}\frac{1}{\left\vert
s_{j}-s_{j-1}\right\vert ^{H(d+2\sum_{l=1}^{d}\alpha _{\lbrack \sigma
(j)]}^{(l)})}}ds_{1}...ds_{2m}.
\end{eqnarray*}

\begin{thm}\label{mainthmlocaltime}
Suppose that $\Psi _{\alpha}^{f}(\theta ,t,z),\Psi _{\alpha}^{\varkappa
}(\theta ,t)<\infty $. Then, defining $\Lambda _{\alpha }^{f}(\theta ,t,z)$
as in \eqref{LambdaDef} gives a random variable in $L^{2}(\Omega )$ and there exists a
universal constant $C=C(T,H,d)>0$ such that%
\begin{equation}
E[\left\vert \Lambda _{\alpha }^{f}(\theta ,t,z)\right\vert ^{2}]\leq
C^{m+\left\vert \alpha\right\vert}\Psi _{\alpha}^{f}(\theta ,t,z).  \label{supestL}
\end{equation}%
Moreover, we have%
\begin{equation}
\left\vert E[\int_{(\mathbb{R}^{d})^{m}}\Lambda _{\alpha }^{f}(\theta
,t,z)dz]\right\vert \leq C^{m/2+\left\vert \alpha\right\vert/2}\prod_{j=1}^{m}\left\Vert
f_{j}\right\Vert _{L^{1}(\mathbb{R}^{d};L^{\infty }([0,T]))}(\Psi
_{\alpha}^{\varkappa }(\theta ,t))^{1/2}.  \label{intestL}
\end{equation}

\end{thm}

\begin{proof}
For notational convenience we consider $\theta =0$ and set $\Lambda _{\alpha
}^{f}(t,z)=\Lambda _{\alpha }^{f}(0,t,z).$

For an integrable function $g:(\mathbb{R}^{d})^{m}\longrightarrow \mathbb{C}$
we can write%
\begin{eqnarray*}
&&\left\vert \int_{(\mathbb{R}^{d})^{m}}g(u_{1},...,u_{m})du_{1}...du_{m}%
\right\vert ^{2} \\
&=&\int_{(\mathbb{R}^{d})^{m}}g(u_{1},...,u_{m})du_{1}...du_{m}\int_{(%
\mathbb{R}^{d})^{m}}\overline{g(u_{m+1},...,u_{2m})}du_{m+1}...du_{2m} \\
&=&\int_{(\mathbb{R}^{d})^{m}}g(u_{1},...,u_{m})du_{1}...du_{m}(-1)^{dm}%
\int_{(\mathbb{R}^{d})^{m}}\overline{g(-u_{m+1},...,-u_{2m})}%
du_{m+1}...du_{2m},
\end{eqnarray*}%
where we used the change of variables $(u_{m+1},...,u_{2m})\longmapsto
(-u_{m+1},...,-u_{2m})$ in the third equality.

This gives%
\begin{eqnarray*}
&&\left\vert \Lambda _{\alpha }^{f}(t,z)\right\vert ^{2} \\
&=&(2\pi )^{-2dm}(-1)^{dm}\int_{(\mathbb{R}^{d})^{2m}}\int_{\Delta
_{0,t}^{m}}\prod_{j=1}^{m}f_{j}(s_{j},z_{j})(-iu_{j})^{\alpha
_{j}}e^{-i\left\langle u_{j},B_{s_{j}}-z_{j}\right\rangle }ds_{1}...ds_{m} \\
&&\times \int_{\Delta
_{0,t}^{m}}\prod_{j=m+1}^{2m}f_{[j]}(s_{j},z_{[j]})(-iu_{j})^{\alpha
_{\lbrack j]}}e^{-i\left\langle u_{j},B_{s_{j}}-z_{[j]}\right\rangle
}ds_{m+1}...ds_{2m}du_{1}...du_{2m} \\
&=&(2\pi )^{-2dm}(-1)^{dm}\sum_{\sigma \in S(m,m)}\int_{(\mathbb{R}%
^{d})^{2m}}\left( \prod_{j=1}^{m}e^{-i\left\langle
z_{j},u_{j}+u_{j+m}\right\rangle }\right) \\
&&\times \int_{\Delta _{0,t}^{2m}}f_{\sigma
}(s,z)\prod_{j=1}^{2m}u_{\sigma (j)}^{\alpha _{\lbrack \sigma
(j)]}}\exp \left\{ -\sum_{j=1}^{2m}\left\langle u_{\sigma
(j)},B_{s_{j}}\right\rangle \right\} ds_{1}...ds_{2m}du_{1}...du_{2m},
\end{eqnarray*}%
where we used (\ref{shuffleIntegral}) in the last step.

Taking the expectation on both sides yields%
\begin{eqnarray}
&&E[\left\vert \Lambda _{\alpha }^{f}(t,z)\right\vert ^{2}]
\label{Lambda} \\
&=&(2\pi )^{-2dm}(-1)^{dm}\sum_{\sigma \in S(m,m)}\int_{(\mathbb{R}%
^{d})^{2m}}\left( \prod_{j=1}^{m}e^{-i\left\langle
z_{j},u_{j}+u_{j+m}\right\rangle }\right)   \notag \\
&&\times \int_{\Delta _{0,t}^{2m}}f_{\sigma
}(s,z)\prod_{j=1}^{2m}u_{\sigma (j)}^{\alpha _{\lbrack \sigma
(j)]}}\exp \left\{ -\frac{1}{2}Var[\sum_{j=1}^{2m}\left\langle u_{\sigma
(j)},B_{s_{j}}\right\rangle ]\right\} ds_{1}...ds_{2m}du_{1}...du_{2m} 
\notag \\
&=&(2\pi )^{-2dm}(-1)^{dm}\sum_{\sigma \in S(m,m)}\int_{(\mathbb{R}%
^{d})^{2m}}\left( \prod_{j=1}^{m}e^{-i\left\langle
z_{j},u_{j}+u_{j+m}\right\rangle }\right)   \notag \\
&&\times \int_{\Delta _{0,t}^{2m}}f_{\sigma
}(s,z)\prod_{j=1}^{2m}u_{\sigma (j)}^{\alpha _{\lbrack \sigma
(j)]}}\exp \left\{ -\frac{1}{2}\sum_{l=1}^{d}Var[\sum_{j=1}^{2m}u_{\sigma
(j)}^{(l)}B_{s_{j}}^{(l)}]\right\}
ds_{1}...ds_{2m}du_{1}^{(1)}...du_{2m}^{(1)}  \notag \\
&&...du_{1}^{(d)}...du_{2m}^{(d)}  \notag \\
&=&(2\pi )^{-2dm}(-1)^{dm}\sum_{\sigma \in S(m,m)}\int_{(\mathbb{R}%
^{d})^{2m}}\left( \prod_{j=1}^{m}e^{-i\left\langle
z_{j},u_{j}+u_{j+m}\right\rangle }\right)   \notag \\
&&\times \int_{\Delta _{0,t}^{2m}}f_{\sigma
}(s,z)\prod_{j=1}^{2m}u_{\sigma (j)}^{\alpha _{\lbrack \sigma
(j)]}}\prod_{l=1}^{d}\exp \left\{ -\frac{1}{2}((u_{\sigma
(j)}^{(l)})_{1\leq j\leq 2m})^{\ast}Q((u_{\sigma (j)}^{(l)})_{1\leq j\leq
2m})\right\} ds_{1}...ds_{2m}  \notag \\
&&du_{\sigma (1)}^{(1)}...du_{\sigma (2m)}^{(1)}...du_{\sigma
(1)}^{(d)}...du_{\sigma (2m)}^{(d)},  \notag
\end{eqnarray}%
where $\ast$ denotes transposition and
\begin{equation*}
Q=Q(s):=(E[B_{s_{i}}^{(1)}B_{s_{j}}^{(1)}])_{1\leq i,j\leq 2m}.
\end{equation*}%
Further, we see that%
\begin{eqnarray}
&&\int_{\Delta _{0,t}^{2m}}\left\vert f_{\sigma }(s,z)\right\vert \int_{(%
\mathbb{R}^{d})^{2m}}\prod_{j=1}^{2m}\prod_{l=1}^{d}\left%
\vert u_{\sigma (j)}^{(l)}\right\vert ^{\alpha _{\lbrack \sigma
(j)]}^{(l)}}\prod_{l=1}^{d}\exp \left\{ -\frac{1}{2}((u_{\sigma
(j)}^{(l)})_{1\leq j\leq 2m})^{\ast}Q((u_{\sigma (j)}^{(l)})_{1\leq j\leq
2m})\right\}   \notag \\
&&du_{\sigma (1)}^{(1)}...du_{\sigma (2m)}^{(1)}...du_{\sigma
(1)}^{(d)}...du_{\sigma (2m)}^{(d)}ds_{1}...ds_{2m}  \notag \\
&=&\int_{\Delta _{0,t}^{2m}}\left\vert f_{\sigma }(s,z)\right\vert \int_{(%
\mathbb{R}^{d})^{2m}}\prod_{j=1}^{2m}\prod_{l=1}^{d}\left%
\vert u_{j}^{(l)}\right\vert ^{\alpha _{\lbrack \sigma (j)]}^{(l)}}  \notag
\\
&&\times \prod_{l=1}^{d}\exp \left\{ -\frac{1}{2}\left\langle
Qu^{(l)},u^{(l)}\right\rangle \right\}   \notag \\
&&du_{1}^{(1)}...du_{2m}^{(1)}...du_{1}^{(d)}...du_{2m}^{(d)}ds_{1}...ds_{2m}
\notag \\
&=&\int_{\Delta _{0,t}^{2m}}\left\vert f_{\sigma }(s,z)\right\vert
\prod_{l=1}^{d}\int_{\mathbb{R}^{2m}}(\prod_{j=1}^{2m}\left%
\vert u_{j}^{(l)}\right\vert ^{\alpha _{\lbrack \sigma (j)]}^{(l)}})\exp
\left\{ -\frac{1}{2}\left\langle Qu^{(l)},u^{(l)}\right\rangle \right\}
du_{1}^{(l)}...du_{2m}^{(l)}ds_{1}...ds_{2m},  \label{Lambda2}
\end{eqnarray}%
where%
\begin{equation*}
u^{(l)}:=(u_{j}^{(l)})_{1\leq j\leq 2m}.
\end{equation*}%
We have that%
\begin{eqnarray*}
&&\int_{\mathbb{R}^{2m}}(\prod_{j=1}^{2m}\left\vert
u_{j}^{(l)}\right\vert ^{\alpha _{\lbrack \sigma (j)]}^{(l)}})\exp \left\{ -%
\frac{1}{2}\left\langle Qu^{(l)},u^{(l)}\right\rangle \right\}
du_{1}^{(l)}...du_{2m}^{(l)} \\
&=&\frac{1}{(\det Q)^{1/2}}\int_{\mathbb{R}^{2m}}(\prod_{j=1}^{2m}%
\left\vert \left\langle Q^{-1/2}u^{(l)},e_{j}\right\rangle \right\vert
^{\alpha _{\lbrack \sigma (j)]}^{(l)}})\exp \left\{ -\frac{1}{2}\left\langle
u^{(l)},u^{(l)}\right\rangle \right\} du_{1}^{(l)}...du_{2m}^{(l)},
\end{eqnarray*}%
where $e_{i},i=1,...,2m$ is the standard ONB of $\mathbb{R}^{2m}$. 

We also get that%
\begin{eqnarray*}
&&\int_{\mathbb{R}^{2m}}(\prod_{j=1}^{2m}\left\vert \left\langle
Q^{-1/2}u^{(l)},e_{j}\right\rangle \right\vert ^{\alpha _{\lbrack \sigma
(j)]}^{(l)}})\exp \left\{ -\frac{1}{2}\left\langle
u^{(l)},u^{(l)}\right\rangle \right\} du_{1}^{(l)}...du_{2m}^{(l)} \\
&=&(2\pi )^{m}E[\prod_{j=1}^{2m}\left\vert \left\langle
Q^{-1/2}Z,e_{j}\right\rangle \right\vert ^{\alpha _{\lbrack \sigma
(j)]}^{(l)}}],
\end{eqnarray*}%
where%
\begin{equation*}
Z\sim \mathcal{N}(\mathcal{O},I_{2m\times 2m}).
\end{equation*}%
We know from Lemma \ref{LiWei}, which is a type of Brascamp-Lieb
inequality that%
\begin{eqnarray*}
&&E[\prod_{j=1}^{2m}\left\vert \left\langle
Q^{-1/2}Z,e_{j}\right\rangle \right\vert ^{\alpha _{\lbrack \sigma
(j)]}^{(l)}}] \\
&\leq &\sqrt{perm(\sum )}=\sqrt{\sum_{\pi \in S_{2\left\vert \alpha
^{(l)}\right\vert }}\prod_{i=1}^{2\left\vert \alpha
^{(l)}\right\vert }a_{i\pi (i)}},
\end{eqnarray*}%
where $perm(\sum )$ is the permanent of the covariance matrix $\sum =(a_{ij})
$ of the Gaussian random vector%
\begin{equation*}
\underset{\alpha _{\lbrack \sigma (1)]}^{(l)}\text{ times}}{\underbrace{%
(\left\langle Q^{-1/2}Z,e_{1}\right\rangle ,...,\left\langle
Q^{-1/2}Z,e_{1}\right\rangle }},\underset{\alpha _{\lbrack \sigma (2)]}^{(l)}%
\text{ times}}{\underbrace{\left\langle Q^{-1/2}Z,e_{2}\right\rangle
,...,\left\langle Q^{-1/2}Z,e_{2}\right\rangle }},...,\underset{\alpha
_{\lbrack \sigma (2m)]}^{(l)}\text{ times}}{\underbrace{\left\langle
Q^{-1/2}Z,e_{2m}\right\rangle ,...,\left\langle
Q^{-1/2}Z,e_{2m}\right\rangle }}),
\end{equation*}%
$\left\vert \alpha ^{(l)}\right\vert :=\sum_{j=1}^{m}\alpha _{j}^{(l)}$ and
where $S_{n}$ stands for the permutation group of size $n$.

In addition, using an upper bound for the permanent of positive semidefinite
matrices (see \cite{AG}) or
direct computations we get that%
\begin{equation}
perm(\sum )=\sum_{\pi \in S_{2\left\vert \alpha ^{(l)}\right\vert
}}\prod_{i=1}^{2\left\vert \alpha ^{(l)}\right\vert }a_{i\pi
(i)}\leq (2\left\vert \alpha ^{(l)}\right\vert
)!\prod_{i=1}^{2\left\vert \alpha ^{(l)}\right\vert }a_{ii}.
\label{PSD}
\end{equation}

Let now $i\in \lbrack \sum_{r=1}^{j-1}\alpha _{\lbrack \sigma
(r)]}^{(l)}+1,\alpha _{\lbrack \sigma (j)]}^{(l)}]$ for some arbitrary fixed 
$j\in \{1,...,2m\}$. Then%
\begin{equation*}
a_{ii}=E[\left\langle Q^{-1/2}Z,e_{j}\right\rangle \left\langle
Q^{-1/2}Z,e_{j}\right\rangle ].
\end{equation*}

\bigskip Further using substitution, we also have that%
\begin{eqnarray*}
&&E[\left\langle Q^{-1/2}Z,e_{j}\right\rangle \left\langle
Q^{-1/2}Z,e_{j}\right\rangle ] \\
&=&(\det Q)^{1/2}\frac{1}{(2\pi )^{m}}\int_{\mathbb{R}^{2m}}\left\langle
u,e_{j}\right\rangle ^{2}\exp (-\frac{1}{2}\left\langle Qu,u\right\rangle
)du_{1}...du_{2m} \\
&=&(\det Q)^{1/2}\frac{1}{(2\pi )^{m}}\int_{\mathbb{R}^{2m}}u_{j}^{2}\exp (-%
\frac{1}{2}\left\langle Qu,u\right\rangle )du_{1}...du_{2m}
\end{eqnarray*}

\bigskip

We now want to use Lemma \ref{CD}. Then we get that%
\begin{eqnarray*}
&&\int_{\mathbb{R}^{2m}}u_{j}^{2}\exp (-\frac{1}{2}\left\langle
Qu,u\right\rangle )du_{1}...du_{m} \\
&=&\frac{(2\pi )^{(2m-1)/2}}{(\det Q)^{1/2}}\int_{\mathbb{R}}v^{2}\exp (-%
\frac{1}{2}v^{2})dv\frac{1}{\sigma _{j}^{2}} \\
&=&\frac{(2\pi )^{m}}{(\det Q)^{1/2}}\frac{1}{\sigma _{j}^{2}},
\end{eqnarray*}%
where $\sigma _{j}^{2}:=Var[B_{s_{j}}^{H}\left\vert
B_{s_{1}}^{H},...,B_{s_{2m}}^{H}\text{ without }B_{s_{j}}^{H}\right] .$

We now want to use strong local non-determinism of the form (see (\ref{eq:SLND})): For all $t\in
\lbrack 0,T],$ $0<r<t:$%
\begin{equation*}
Var[B_{t}^{H}\left\vert B_{s}^{H},\left\vert t-s\right\vert \geq r\right]
\geq Kr^{2H}.
\end{equation*}%
The latter implies that 
\begin{equation*}
(\det Q(s))^{1/2}\geq K^{(2m-1)/2}\left\vert s_{1}\right\vert ^{H}\left\vert
s_{2}-s_{1}\right\vert ^{H}...\left\vert s_{2m}-s_{2m-1}\right\vert ^{H}
\end{equation*}%
as well as%
\begin{equation*}
\sigma _{j}^{2}\geq K\min \{\left\vert s_{j}-s_{j-1}\right\vert
^{2H},\left\vert s_{j+1}-s_{j}\right\vert ^{2H}\}.
\end{equation*}%
Thus%
\begin{eqnarray*}
\prod_{j=1}^{2m}\sigma _{j}^{-2\alpha _{\lbrack \sigma (j)]}^{(l)}}
&\leq &K^{-2m}\prod_{j=1}^{2m}\frac{1}{\min \{\left\vert
s_{j}-s_{j-1}\right\vert ^{2H\alpha _{\lbrack \sigma (j)]}^{(l)}},\left\vert
s_{j+1}-s_{j}\right\vert ^{2H\alpha _{\lbrack \sigma (j)]}^{(l)}}\}} \\
&\leq &C^{m+\left\vert \alpha ^{(l)}\right\vert}\prod_{j=1}^{2m}\frac{1}{\left\vert
s_{j}-s_{j-1}\right\vert ^{4H\alpha _{\lbrack \sigma (j)]}^{(l)}}}
\end{eqnarray*}%
for a constant $C$ only depending on $H$ and $T$.

\bigskip Hence, it follows from (\ref{PSD}) that%
\begin{eqnarray*}
perm(\sum ) &\leq &(2\left\vert \alpha ^{(l)}\right\vert
)!\prod_{i=1}^{2\left\vert \alpha ^{(l)}\right\vert }a_{ii} \\
&\leq &(2\left\vert \alpha ^{(l)}\right\vert
)!\prod_{j=1}^{2m}((\det Q)^{1/2}\frac{1}{(2\pi )^{m}}\frac{(2\pi
)^{m}}{(\det Q)^{1/2}}\frac{1}{\sigma _{j}^{2}})^{\alpha _{\lbrack \sigma
(j)]}^{(l)}} \\
&\leq &(2\left\vert \alpha ^{(l)}\right\vert )!C^{m+\left\vert \alpha ^{(l)}\right\vert}\prod_{j=1}^{2m}%
\frac{1}{\left\vert s_{j}-s_{j-1}\right\vert ^{4H\alpha _{\lbrack \sigma
(j)]}^{(l)}}}.
\end{eqnarray*}%
So%
\begin{eqnarray*}
&&E[\prod_{j=1}^{2m}\left\vert \left\langle
Q^{-1/2}Z,e_{j}\right\rangle \right\vert ^{\alpha _{\lbrack \sigma
(j)]}^{(l)}}]\leq \sqrt{perm(\sum )} \\
&\leq &\sqrt{(2\left\vert \alpha ^{(l)}\right\vert )!}C^{m+\left\vert \alpha ^{(l)}\right\vert}\prod_{j=1}^{2m}\frac{1}{\left\vert s_{j}-s_{j-1}\right\vert ^{2H\alpha
_{\lbrack \sigma (j)]}^{(l)}}}.
\end{eqnarray*}%
Therefore we obtain from (\ref{Lambda}) and (\ref{Lambda2}) that%
\begin{eqnarray*}
&&E[\left\vert \Lambda _{\alpha }^{f}(\theta ,t,z)\right\vert ^{2}] \\
&\leq &C^{m}\sum_{\sigma\in S(m,m)}\int_{\Delta _{0,t}^{2m}}\left\vert f_{\sigma }(s,z)\right\vert
\prod_{l=1}^{d}\int_{\mathbb{R}^{2m}}(\prod_{j=1}^{2m}\left%
\vert u_{j}^{(l)}\right\vert ^{\alpha _{\lbrack \sigma (j)]}^{(l)}})\exp
\left\{ -\frac{1}{2}\left\langle Qu^{(l)},u^{(l)}\right\rangle \right\}
du_{1}^{(l)}...du_{2m}^{(l)}ds_{1}...ds_{2m} \\
&\leq &M^{m}\sum_{\sigma\in S(m,m)}\int_{\Delta _{0,t}^{2m}}\left\vert f_{\sigma }(s,z)\right\vert 
\frac{1}{(\det Q(s))^{d/2}}\prod_{l=1}^{d}\sqrt{(2\left\vert \alpha
^{(l)}\right\vert )!}C^{m+\left\vert \alpha ^{(l)}\right\vert}\prod_{j=1}^{2m}\frac{1}{\left\vert
s_{j}-s_{j-1}\right\vert ^{2H\alpha _{\lbrack \sigma (j)]}^{(l)}}}%
ds_{1}...ds_{2m} \\
&=&M^{m}C^{md+\left\vert \alpha\right\vert}\prod_{l=1}^{d}\sqrt{(2\left\vert \alpha
^{(l)}\right\vert )!}\sum_{\sigma\in S(m,m)}\int_{\Delta _{0,t}^{2m}}\left\vert f_{\sigma
}(s,z)\right\vert \prod_{j=1}^{2m}\frac{1}{\left\vert
s_{j}-s_{j-1}\right\vert ^{H(d+2\sum_{l=1}^{d}\alpha _{\lbrack \sigma
(j)]}^{(l)})}}ds_{1}...ds_{2m}
\end{eqnarray*}%
for a constant $M$ depending on $d$.

\bigskip Finally, we show estimate (\ref{intestL}). Using the inequality (\ref{supestL}), we find that%
\begin{eqnarray*}
&&\left\vert E\left[ \int_{(\mathbb{R}^{d})^{m}}\Lambda _{\alpha
}^{\varkappa f}(\theta ,t,z)dz\right] \right\vert \\
&\leq &\int_{(\mathbb{R}^{d})^{m}}(E[\left\vert \Lambda _{\alpha
}^{\varkappa f}(\theta ,t,z)\right\vert ^{2})^{1/2}dz\leq C^{m/2+\left\vert \alpha\right\vert/2}\int_{(%
\mathbb{R}^{d})^{m}}(\Psi _{\alpha}^{\varkappa f}(\theta ,t,z))^{1/2}dz.
\end{eqnarray*}%
Taking the supremum over $[0,T]$ for each function $f_{j}$, i.e.%
\begin{equation*}
\left\vert f_{[\sigma (j)]}(s_{j},z_{[\sigma (j)]})\right\vert \leq
\sup_{s_{j}\in \lbrack 0,T]}\left\vert f_{[\sigma (j)]}(s_{j},z_{[\sigma
(j)]})\right\vert ,j=1,...,2m
\end{equation*}%
one obtains that%
\begin{eqnarray*}
&&\left\vert E\left[ \int_{(\mathbb{R}^{d})^{m}}\Lambda _{\alpha
}^{\varkappa f}(\theta ,t,z)dz\right] \right\vert \\
&\leq &C^{m+\left\vert \alpha\right\vert}\max_{\sigma \in S(m,m)}\int_{(\mathbb{R}^{d})^{m}}\left(
\prod_{l=1}^{2m}\left\Vert f_{[\sigma (l)]}(\cdot ,z_{[\sigma
(l)]})\right\Vert _{L^{\infty }([0,T])}\right) ^{1/2}dz \\
&&\times (\prod_{l=1}^{d}\sqrt{(2\left\vert \alpha ^{(l)}\right\vert
)!}\sum_{\sigma \in S(m,m)}\int_{\Delta _{0,t}^{2m}}\left\vert \varkappa
_{\sigma }(s)\right\vert \prod_{j=1}^{2m}\frac{1}{\left\vert
s_{j}-s_{j-1}\right\vert ^{H(d+2\sum_{l=1}^{d}\alpha _{\lbrack \sigma
(j)]}^{(l)})}}ds_{1}...ds_{2m})^{1/2} \\
&=&C^{m+\left\vert \alpha\right\vert}\max_{\sigma \in S(m,m)}\int_{(\mathbb{R}^{d})^{m}}\left(
\prod_{l=1}^{2m}\left\Vert f_{[\sigma (l)]}(\cdot ,z_{[\sigma
(l)]})\right\Vert _{L^{\infty }([0,T])}\right) ^{1/2}dz\cdot (\Psi
_{\alpha}^{\varkappa }(\theta ,t))^{1/2} \\
&=&C^{m+\left\vert \alpha\right\vert}\int_{(\mathbb{R}^{d})^{m}}\prod_{j=1}^{m}\left\Vert
f_{j}(\cdot ,z_{j})\right\Vert _{L^{\infty }([0,T])}dz\cdot (\Psi
_{\alpha}^{\varkappa }(\theta ,t))^{1/2} \\
&=&C^{m+\left\vert \alpha\right\vert}\prod_{j=1}^{m}\left\Vert f_{j}(\cdot ,z_{j})\right\Vert
_{L^{1}(\mathbb{R}^{d};L^{\infty }([0,T]))}\cdot (\Psi _{\alpha}^{\varkappa
}(\theta ,t))^{1/2}.
\end{eqnarray*}

\end{proof}

We remark that \emph{a priori} one can not interchange the order of integration in (\ref{LambdaDef}). Indeed, for $m=1$, $f \equiv 1$ one gets an integral of the Donsker-Delta function which is not a random variable in the usual sense. 
To overcome this define for $R>0$, 
$$
\Lambda^f_{\alpha,R} (\theta,t,z) := (2\pi)^{-dm} \int_{B(0,R)} \int_{\Delta^m_{\theta,t}} \prod_{j=1}^m f_j(s_j,z_j) (-iu_j)^{\alpha_j} e^{-i \langle u_j, B_{s_j}-z_j\rangle } ds du,
$$
where $B(0,R):=\{v\in (\R^{d})^m : \, |v|<R\}$. Clearly we have
$$
|\Lambda^f_{\alpha,R} (\theta,t,z)| \leq C_{R} \int_{\Delta^m_{\theta,t}} \prod_{j=1}^m |f_j(s_j,z_j)| ds 
$$
for an appropriate constant $C_R$. Let us assume that the above right-hand side is integrable over $(\R^d)^m$. 

Similar computations as above show that $\Lambda^f_{\alpha,R}(\theta,t,z)\to \Lambda^f_{\alpha} (\theta,t,z)$ in $L^2(\Omega)$ as $R\to \infty$ for all $\theta,t$ and $z$.

Lebesgue's dominated convergence theorem and the fact that the Fourier transform is an automorphism on the Schwarz space yield
\begin{align*}
\int_{(\R^d)^m}&  \Lambda^f_{\alpha}(\theta,t,z) dz = \lim_{R\to \infty} \int_{(\R^d)^m}  \Lambda^f_{\alpha, R} (\theta,t,z)dz\\
&= \lim_{R\to \infty} (2\pi)^{-dm}\int_{(\R^d)^m} \int_{B(0,R)} \int_{\Delta_{\theta,t}^m}  \prod_{j=1}^m f_j(s_j,z_j) (-iu_j)^{\alpha_j} e^{-i  \langle u_j, B_{s_j}-z_j \rangle_{\R^d}} dsdudz\\
&= \lim_{R\to \infty}   \int_{\Delta_{\theta,t}^m}   \int_{B(0,R)} (2\pi)^{-dm} \int_{(\R^d)^m}   \prod_{j=1}^m f_j(s_j,z_j)e^{i u_j z_j}  (-iu_j)^{\alpha_j} e^{-i  \langle u_j, B_{s_j}\rangle_{\R^d}} \, dzduds\\
&= \lim_{R\to \infty} \int_{\Delta_{\theta,t}^m}    \int_{B(0,R)}  \prod_{j=1}^m \widehat{f}_j(s,-u_j)  (-iu_j)^{\alpha_j} e^{-i \langle u_j, B_{s_j}\rangle_{\R^d}} duds\\
&= \int_{\Delta_{\theta,t}^m} D^{\alpha}f(s,B_s)ds
\end{align*}
which is exactly \eqref{ibp}.

Next, we give a crucial estimate which shows why fractional Brownian motion actually regularises \eqref{VI_SDE}. It is based on integration by parts and the aforementioned properties of the local-time $L$. The estimate we obtain can be presented in a more explicit way when
\begin{align*}
\varkappa_j(s) = (K_H(s,\theta)-K_H(s,\theta'))^{\varepsilon_j}, \quad \theta < s < t
\end{align*}
or,
\begin{align*}
\varkappa_j(s) = (K_H(s,\theta))^{\varepsilon_j}, \quad \theta < s < t
\end{align*}
for every $j=1,\dots,m$ with $(\varepsilon_1,\dots, \varepsilon_{m})\in \{0,1\}^{m}$ and we will see why these choices are important in the next section.

\begin{prop}\label{mainestimate1}
Let $B^{H},H\in (0,1/2)$ be a standard $d-$dimensional fractional Brownian
motion and functions $f$ and $\varkappa $ as in (\ref{f}), respectively
as in (\ref{kappa}). Let $\theta ,\theta \prime ,t\in \lbrack 0,T],\theta
\prime <\theta <t$ and%
\begin{equation*}
\varkappa _{j}(s)=(K_{H}(s,\theta )-K_{H}(s,\theta \prime ))^{\varepsilon
_{j}},\theta <s<t
\end{equation*}
for every $j=1,...,m$ with $(\varepsilon _{1},...,\varepsilon _{m})\in
\{0,1\}^{m}$ for $\theta ,\theta \prime \in \lbrack 0,T]$ with $\theta
\prime <\theta .$ Let $\alpha \in (\mathbb{N}_{0}^{d})^{m}$ be a
multi-index. If 
\begin{align*}\label{3.7}
H<\frac{\frac{1}{2}-\gamma }{(d-1+2\sum_{l=1}^{d}\alpha _{
j}^{(l)})}
\end{align*}%
for all $j$, where $\gamma \in (0,H)$ is sufficiently small, then there exists a universal constant $C$ (depending on $H$, $T$
and $d$, but independent of $m$, $\{f_{i}\}_{i=1,...,m}$ and $\alpha $) such
that for any $\theta ,t\in \lbrack 0,T]$ with $\theta <t$ we have%
\begin{eqnarray*}
&&\left\vert E\int_{\Delta _{\theta ,t}^{m}}\left(
\prod_{j=1}^{m}D^{\alpha _{j}}f_{j}(s_{j},B_{s_{j}}^{H})\varkappa
_{j}(s_{j})\right) ds\right\vert  \\
&\leq &C^{m+\left\vert \alpha\right\vert}\prod_{j=1}^{m}\left\Vert f_{j}(\cdot ,z_{j})\right\Vert
_{L^{1}(\mathbb{R}^{d};L^{\infty }([0,T]))}\left( \frac{\theta -\theta
\prime }{\theta \theta \prime }\right) ^{\gamma \sum_{j=1}^{m}\varepsilon
_{j}}\theta ^{(H-\frac{1}{2}-\gamma )\sum_{j=1}^{m}\varepsilon _{j}} \\
&&\times \frac{(\prod_{l=1}^{d}(2\left\vert \alpha ^{(l)}\right\vert
)!)^{1/4}(t-\theta )^{-H(md+2\left\vert \alpha \right\vert )-(H-\frac{1}{2}%
-\gamma )\sum_{j=1}^{m}\varepsilon _{j}+m}}{\Gamma (-H(2md+4\left\vert
\alpha \right\vert )+2(H-\frac{1}{2}-\gamma )\sum_{j=1}^{m}\varepsilon
_{j}+2m)^{1/2}}.
\end{eqnarray*}%

\end{prop}

\begin{proof}
By definition of $\Lambda _{\alpha }^{\varkappa f}$ (\ref{LambdaDef}) it
immediately follows that the integral in our proposition can be expressed as%
\begin{equation*}
\int_{\Delta _{\theta ,t}^{m}}\left( \prod_{j=1}^{m}D^{\alpha
_{j}}f_{j}(s_{j},B_{s_{j}}^{H})\varkappa _{j}(s_{j})\right) ds=\int_{\mathbb{%
R}^{dm}}\Lambda _{\alpha }^{\varkappa f}(\theta ,t,z)dz.
\end{equation*}%
Taking expectation and using Theorem \ref{mainthmlocaltime} we obtain%
\begin{equation*}
\left\vert E\int_{\Delta _{\theta ,t}^{m}}\left(
\prod_{j=1}^{m}D^{\alpha _{j}}f_{j}(s_{j},B_{s_{j}}^{H})\varkappa
_{j}(s_{j})\right) ds\right\vert \leq C^{m+\left\vert \alpha \right\vert}\prod_{j=1}^{m}\left\Vert
f_{j}(\cdot ,z_{j})\right\Vert _{L^{1}(\mathbb{R}^{d};L^{\infty
}([0,T]))}\cdot (\Psi _{\alpha}^{\varkappa }(\theta ,t))^{1/2},
\end{equation*}%
where in this situation 
\begin{eqnarray*}
&&\Psi _{k}^{\varkappa }(\theta ,t) \\
&:&=\prod_{l=1}^{d}\sqrt{(2\left\vert \alpha ^{(l)}\right\vert )!}%
\sum_{\sigma \in S(m,m)}\int_{\Delta
_{0,t}^{2m}}\prod_{j=1}^{2m}(K_{H}(s_{j},\theta )-K_{H}(s_{j},\theta
\prime ))^{\varepsilon _{\lbrack \sigma (j)]}} \\
&&\frac{1}{\left\vert s_{j}-s_{j-1}\right\vert ^{H(d+2\sum_{l=1}^{d}\alpha
_{\lbrack \sigma (j)]}^{(l)})}}ds_{1}...ds_{2m}.
\end{eqnarray*}%
We want to apply Lemma \ref{VI_iterativeInt}. For this, we need that $%
-H(d+2\sum_{l=1}^{d}\alpha _{\lbrack \sigma (j)]}^{(l)})+(H-\frac{1}{2}%
-\gamma )\varepsilon _{\lbrack \sigma (j)]}>-1$ for all $j=1,...,2m.$ The
worst case is, when $\varepsilon _{\lbrack \sigma (j)]}=1$ for all $j$. So $%
H<\frac{\frac{1}{2}-\gamma }{(d-1+2\sum_{l=1}^{d}\alpha _{\lbrack \sigma
(j)]}^{(l)})}$ for all $j$.
Hence, we have%
\begin{eqnarray*}
\Psi _{\alpha}^{\varkappa }(\theta ,t) &\leq &\sum_{\sigma \in S(m,m)}\left( 
\frac{\theta -\theta \prime }{\theta \theta \prime }\right) ^{\gamma
\sum_{j=1}^{2m}\varepsilon _{\lbrack \sigma (j)]}}\theta ^{(H-\frac{1}{2}%
-\gamma )\sum_{j=1}^{2m}\varepsilon _{\lbrack \sigma (j)]}} \\
&&\times \prod_{l=1}^{d}\sqrt{(2\left\vert \alpha ^{(l)}\right\vert
)!}\Pi _{\gamma }(2m)(t-\theta )^{-H(2md+4\left\vert \alpha \right\vert )+(H-%
\frac{1}{2}-\gamma )\sum_{j=1}^{2m}\varepsilon _{\lbrack \sigma (j)]}+2m},
\end{eqnarray*}%
where $\Pi _{\gamma }(m)$ is defined as in Lemma \ref{VI_iterativeInt}. The latter can be bounded above as follows
\begin{equation*}
\Pi_{\gamma }(2m)\leq \frac{\prod_{j=1}^{2m}\Gamma
(1-H(d+2\sum_{l=1}^{d}\alpha _{\lbrack \sigma (j)]}^{(l)}))}{\Gamma
(-H(2md+4\left\vert \alpha \right\vert )+(H-\frac{1}{2}-\gamma
)\sum_{j=1}^{2m}\varepsilon _{\lbrack \sigma (j)]}+2m)}.
\end{equation*}%
Observe that $\sum_{j=1}^{2m}\varepsilon _{\lbrack \sigma
(j)]}=2\sum_{j=1}^{m}\varepsilon _{j}.$ Therefore, we have that%
\begin{eqnarray*}
&&(\Psi _{k}^{\varkappa }(\theta ,t))^{1/2} \\
&\leq &C^{m}\left( \frac{\theta -\theta \prime }{\theta \theta \prime }%
\right) ^{\gamma \sum_{j=1}^{m}\varepsilon _{j}}\theta ^{(H-\frac{1}{2}%
-\gamma )\sum_{j=1}^{m}\varepsilon _{j}} \\
&&\times \frac{(\prod_{l=1}^{d}(2\left\vert \alpha ^{(l)}\right\vert
)!)^{1/4}(t-\theta )^{-H(md+2\left\vert \alpha \right\vert )-(H-\frac{1}{2}%
-\gamma )\sum_{j=1}^{m}\varepsilon _{j}+m}}{\Gamma (-H(2md+4\left\vert
\alpha \right\vert )+2(H-\frac{1}{2}-\gamma )\sum_{j=1}^{m}\varepsilon
_{j}+2m)^{1/2}},
\end{eqnarray*}%
where we used $\prod_{j=1}^{2m}\Gamma (1-H(d+2\sum_{l=1}^{d}\alpha _{\lbrack \sigma (j)]}^{(1)})\leq C^{m}$ for a large
enough constant $C>0$ and $\sqrt{a_{1}+...+a_{m}}\leq \sqrt{a_{1}}+...\sqrt{%
a_{m}}$ for arbitrary non-negative numbers $a_{1},...,a_{m}$.

\end{proof}

\begin{prop}\label{mainestimate2}
Let $B^{H},H\in (0,1/2)$ be a standard $d-$dimensional fractional Brownian
motion and functions $f$ and $\varkappa $ as in (\ref{f}), respectively
as in (\ref{kappa}). Let $\theta ,t\in \lbrack 0,T]$ with $\theta <t$ and%
\begin{equation*}
\varkappa _{j}(s)=(K_{H}(s,\theta ))^{\varepsilon _{j}},\theta <s<t
\end{equation*}%
for every $j=1,...,m$ with $(\varepsilon _{1},...,\varepsilon _{m})\in
\{0,1\}^{m}$. Let $\alpha \in (\mathbb{N}_{0}^{d})^{m}$ be a
multi-index. If 
\begin{align*}
H<\frac{\frac{1}{2}-\gamma }{(d-1+2\sum_{l=1}^{d}\alpha _{
j}^{(l)})}
\end{align*}%
for all $j$, where $\gamma \in (0,H)$ is sufficiently small, then there exists a universal constant $C$ (depending on $H$, $T$
and $d$, but independent of $m$, $\{f_{i}\}_{i=1,...,m}$ and $\alpha $) such
that for any $\theta ,t\in \lbrack 0,T]$ with $\theta <t$ we have%
\begin{eqnarray*}
&&\left\vert E\int_{\Delta _{\theta ,t}^{m}}\left(
\prod_{j=1}^{m}D^{\alpha _{j}}f_{j}(s_{j},B_{s_{j}}^{H})\varkappa
_{j}(s_{j})\right) ds\right\vert  \\
&\leq &C^{m+\left\vert \alpha\right\vert}\prod_{j=1}^{m}\left\Vert f_{j}(\cdot ,z_{j})\right\Vert
_{L^{1}(\mathbb{R}^{d};L^{\infty }([0,T]))}\theta ^{(H-\frac{1}{2})\sum_{j=1}^{m}\varepsilon _{j}} \\
&&\times \frac{(\prod_{l=1}^{d}(2\left\vert \alpha ^{(l)}\right\vert
)!)^{1/4}(t-\theta )^{-H(md+2\left\vert \alpha \right\vert )-(H-\frac{1}{2}%
-\gamma )\sum_{j=1}^{m}\varepsilon _{j}+m}}{\Gamma (-H(2md+4\left\vert
\alpha \right\vert )+2(H-\frac{1}{2}-\gamma )\sum_{j=1}^{m}\varepsilon
_{j}+2m)^{1/2}}.
\end{eqnarray*}%

\end{prop}

\begin{proof}
The proof is similar to the previous proposition.
\end{proof}

\begin{rem}\label{Remark 3.4}
We mention that%
\begin{equation*}
\prod_{l=1}^{d}(2\left\vert \alpha ^{(l)}\right\vert )!\leq
(2\left\vert \alpha \right\vert )!C^{\left\vert \alpha \right\vert }
\end{equation*}%
for a constant $C$ depending on $d$. Later on in the paper, when we deal
with the existence of strong solutions, we will consider the case%
\begin{equation*}
\alpha _{j}^{(l)}\in \{0,1\}\text{ for all }j,l
\end{equation*}%
with%
\begin{equation*}
\left\vert \alpha \right\vert =m.
\end{equation*}

\end{rem}


\section{Existence and uniqueness of global strong solutions}

As outlined in the introduction the object of study is a time-inhomogeneous SDE with additive $d$-dimensional fractional Brownian noise $B^H$ with Hurst parameter $H\in (0,1/2)$, i.e.
\begin{align}
dX_t = b(t,X_t)dt + dB_t^H, \quad X_0 = x\in \R^d, \quad t\in [0,T],
\end{align}
where $b:[0,T]\times \R^d\rightarrow \R^d$ is a Borel-measurable function. We will study equation \eqref{VI_SDE} when the drift coefficient $b$ belongs to $L^1(\R^d; L^\infty([0,T],\R^d)) \cap L^\infty(\R^d ; L^\infty ([0,T],\R^d))$. We will introduce the following short notation for the following functional spaces
\begin{align*}
L_{\infty}^{1} &:= L^1(\R^d; L^\infty([0,T],\R^d))\\
L_{\infty}^{\infty} &:= L^\infty(\R^d ; L^\infty ([0,T],\R^d))\\
L_{\infty,\infty}^{1,\infty} &:= L_{\infty}^{1}\cap L_{\infty}^{\infty}.
\end{align*}
Hence the subscript refers to the \emph{supremum}'s norm on $[0,T]$ whereas the superscript indicates the norm used for the space variable.

Hereunder, we establish the main result of this section.

\begin{thm}\label{VI_mainthm}
Let $b\in L_{\infty,\infty}^{1,\infty}$. Then if $H<\frac{1}{2(2+d)}$, $d\geq 1$ there exists a unique (global) strong solution $X=\{X_t, t\in [0,T]\}$ of equation \eqref{VI_SDE}. Moreover, for every $t\in [0,T]$, $X_t$ is Malliavin differentiable in the direction of the Brownian motion $W$ in \eqref{VI_WBH}.
\end{thm}

The proof of Theorem \ref{VI_mainthm} is based on the following steps:

\begin{enumerate}
\item First, we construct a weak solution $X$ to (\ref{VI_SDE}) by means of Girsanov's theorem, that is we introduce a probability space $(\Omega, \mathfrak{A}, P)$ that carries a fractional Brownian motion $B^H$ and a process $X$ such that (\ref{VI_SDE}) is fulfilled. However, a priori $X$ is not adapted to the filtration $\mathcal{F}=\{\mathcal{F}_t\}_{t\in[0,T]}$ generated by $B^H$.

\item Next, we approximate the drift coefficient $b$ a.e. by a sequence of functions (which always exists by standard approximation results) $b_n\subset C_c^{\infty}([0,T]\times \R^d)$, $n\geq 0$ (actually it suffices to look at approximating coefficients which are only smooth with respect to the space variable) such that
\begin{align}\label{approxb}
b_n(t,x) \longrightarrow b(t,x)
\end{align}
as $n\to \infty$ for a.e. $(t,x)\in [0,T]\times \R^d$ with $\sup_{n\geq 0}\|b_n\|_{L_{\infty}^{1}}<\infty$ and such that $|b_n(t,x)|\leq M<\infty$, $n\geq 0$ a.e. for some constant $M$. By standard results on SDEs, we know that for each smooth coefficient $b_n$, $n\ge 0$, there exists unique strong solution $X_\cdot^{n}$ to the SDE  
\begin{align}\label{VI_Xn}
dX_t^{n} = b_n(t,X_t^{n}) du + dB_t^H, \,\,0\leq t\leq T,\,\,\, X_0^{n}=x  \in \mathbb{R}^d\,.
\end{align}
We then show that for each $t\in[0,T]$ the sequence $X_t^{n}$ converges weakly to the conditional expectation $\textrm{ E}[X_t|\mathcal{F}_t]$ in the space $L^2(\Omega;\mathcal{F}_t)$ of square integrable, \mbox{$\mathcal{F}_t$-measurable} random variables. 
\item It is well known, see e.g. \cite{Nua10}, that for each $t\in[0,T]$ the strong solution $X_t^{n}$, $n\ge 0$, is Malliavin differentiable, and that the Malliavin derivative $D_s X_t^{n}$, $0\leq s\leq t$, with respect to $W$ in \eqref{VI_WBH} satisfies
\begin{align}\label{VI_DXn1}
D_s X_t^{n} = K_H(t,s) I_{d} + \int_s^t b_n'(u,X_u^{n})D_s X_u^{n} du, 
\end{align}
where $b_n'$ denotes the Jacobian of $b_n$ and $I_d$ the identity matrix in $\R^{d\times d}$. In the next step we then employ a compactness criterion based on Malliavin calculus to show that for every $t\in [0,T]$ the set of random variables $\{X_t^{n}\}_{n \geq 0}$ is relatively compact in $L^2(\Omega)$, which then admits the conclusion that $X_t^{n}$ converges strongly in $L^2(\Omega;\mathcal{F}_t)$ to $\textrm{ E}[X_t|\mathcal{F}_t]$. Further we see that $\textrm{ E}[X_t|\mathcal{F}_t]$ is Malliavin differentiable as a consequence of the compactness criterion.
\item In the last step we show that $\textrm{ E}[X_t|\mathcal{F}_t] = X_t$, which implies that $X_t$ is $\mathcal{F}_t$-measurable and thus a strong solution on our specific probability space.
\end{enumerate}

We turn to the first step of our scheme which is to construct weak solutions of \eqref{VI_SDE} by using Girsanov's theorem in this context. Let $(\Omega, \mathfrak{A}, \widetilde{P})$ be some given probability space which carries a $d$-dimensional fractional Brownian motion $\widetilde{B}^H$ with Hurst parameter $H\in (0,1/2)$ and set $X_t := x+ \widetilde{B}_t^H$, \mbox{$t\in [0,T]$}, $x\in \R^d$. Set $\theta_t := \left(K_H^{-1}\left(\int_0^{\cdot} b(r,X_r)dr\right) \right) (t)$ and consider the Dol\'{e}ans-Dade exponential
\begin{align*}
\xi_t &:= \mathcal{E}\left( \theta_{\cdot}\right)_t:= \exp \left\{ \int_0^t \theta_s^T dW_s -\frac{1}{2} \int_0^t \theta_s^T \theta_s ds  \right\}, \quad t\in [0,T].
\end{align*}

The following two lemmata show that the conditions of Theorem \ref{VI_girsanov} hold.

\begin{lem}\label{VI_novikov0}
Let $\tilde{B}_t^H$ be a $d$-dimensional fractional Brownian motion with respect to $(\Omega, \mathfrak{A}, \tilde{P})$. Then
$$\int_0^{\cdot} |b(s,\tilde{B}_s^H)| ds \in I_{0^+}^{H+\frac{1}{2}} (L^2), \quad P-a.s.$$
\end{lem}
\begin{proof}
Using the property that $D_{0^+}^{H+\frac{1}{2}} I _{0^+}^{H+\frac{1}{2}} (f) = f$ for $f\in L^2 ([0,T])$ we need to show that
$$D_{0^+}^{H+\frac{1}{2}}\int_0^{\cdot} |b(s,\tilde{B}_s^H)| ds\in L^2([0,T]), \quad P-a.s.$$
Indeed,
\begin{align*}
 \left|D_{0^+}^{H+\frac{1}{2}} \left(\int_0^{\cdot} |b(s,\tilde{B}_s^H)| ds\right)(t) \right|  =& \frac{1}{\Gamma\left(\frac{1}{2}-H\right)}\Bigg( \frac{1}{t^{H+\frac{1}{2}}} \int_0^t |b(u,\widetilde{B}_u^H)|du \\
 &+ \, \left( H +\frac{1}{2}\right) \int_0^t  (t-s)^{-H-\frac{3}{2}}\int_s^t |b(u,\widetilde{B}_u^H)|du  ds\Bigg)\\
 \leq & \, \frac{t^{\frac{1}{2}-H}}{\Gamma\left(\frac{1}{2}-H\right)}\frac{1}{\frac{1}{2}-H}\|b\|_{L_{\infty}^{\infty}}.
\end{align*}
Hence, for some finite constant $C_H>0$ we have
$$\left|D_{0^+}^{H+\frac{1}{2}} \left(\int_0^{\cdot} |b(s,\tilde{B}_s^H)| ds\right)(t) \right|^2 \leq C_H \|b\|_{L_{\infty}^{\infty}}^2 t^{1-2H}.$$
As a result,
$$\int_0^T \left|D_{0^+}^{H+\frac{1}{2}} \left(\int_0^{\cdot} |b(s,\tilde{B}_s^H)| ds\right)(t) \right|^2 dt \leq C_H \|b\|_{L_{\infty}^\infty}^2 \int_0^T t^{1-2H}dt<\infty, \quad P-a.s.$$
since $H\in (0,1/2)$.
\end{proof}

\begin{lem}\label{VI_novikov}
Let $\tilde{B}_t^H$ be a $d$-dimensional fractional Brownian motion with respect to $(\Omega, \mathfrak{A}, \tilde{P})$. Then for every $\mu \in \R$ we have
$$\tilde{\textrm{\emph{E}}}\left[ \exp\left\{\mu \int_0^T \left|K_H^{-1}\left( \int_0^{\cdot} b(r,\tilde{B}_r^H) dr\right) (s)\right|^2 ds\right\} \right] \leq C_{H,d,\mu,T}(\|b\|_{L_{\infty}^{\infty}})$$
for some continuous increasing function $C_{H,d,\mu,T}$ depending only on $H$, $d$, $T$ and $\mu$.

In particular,
$$\tilde{\textrm{\emph{E}}}\left[ \mathcal{E}\left(\int_0^T K_H^{-1}\left( \int_0^{\cdot} b(r,\tilde{B}_r^H) dr\right)^{\ast} (s) dW_s\right)^p \right] \leq C_{H,d,\mu,T}(\|b\|_{L_{\infty}^{\infty}}),$$
where $\tilde{\textrm{\emph{E}}}$ denotes expectation under $\tilde{P}$ and $\ast$ denotes transposition.
\end{lem}
\begin{proof}
Denote by $\theta_s := K_H^{-1}\left( \int_0^{\cdot} |b(r,\tilde{B}_r^H)| dr\right) (s)$. Then using relation \eqref{VI_inverseKH} we have
\begin{align*}
|\theta_s| =& |s^{H-\frac{1}{2}} I_{0^+}^{\frac{1}{2}-H} s^{\frac{1}{2}-H} |b(s,\widetilde{B}_s^H)||\\
=& \frac{1}{\Gamma \left(\frac{1}{2}-H\right)} s^{H- \frac{1}{2}} \int_0^s (s-r)^{-\frac{1}{2}-H} r^{\frac{1}{2}-H} |b(r,\widetilde{B}_r^H)|dr\\
\leq & \, \|b\|_{L_{\infty}^{\infty}}\frac{1}{\Gamma \left(\frac{1}{2}-H\right)} s^{H- \frac{1}{2}} \int_0^s (s-r)^{-\frac{1}{2}-H} r^{\frac{1}{2}-H}dr\\
=& \, \|b\|_{L_{\infty}^{\infty}} \frac{\Gamma \left(\frac{3}{2}-H\right)}{\Gamma \left(1-2H \right)}s^{\frac{1}{2}-H}\\
\leq &\, \|b\|_{L_{\infty}^{\infty}} \frac{\Gamma \left(\frac{3}{2}-H\right)}{\Gamma \left(1-2H \right)}T^{\frac{1}{2}-H}.
\end{align*}

Squaring both sides we have the following estimate
\begin{align}\label{VI_fracL2}
|\theta_s|^2 \leq C_H \|b\|_{L_{\infty}^{\infty}}^2 T^{1-2H} \quad P-a.s.,
\end{align}
where $C_H:=\frac{\Gamma \left(\frac{3}{2}-H\right)^2}{\Gamma \left(1-2H \right)^2}$.

Then we get the following estimate
\begin{align*}
\tilde{\textrm{E}}&\left[ \exp\left\{\mu \int_0^T \left|\theta_s\right|^2 ds\right\} \right] \leq \exp \left\{ |\mu| C_H T^{2(1-H)} \|b\|_{L_{\infty}^{\infty}}^2\right\}.
\end{align*}
\end{proof}

By Girsanov's theorem, see Theorem \ref{VI_girsanov}, the process
\begin{align}\label{VI_weak}
B_t^H := X_t - x - \int_0^t b(s,X_s)ds, \quad t\in [0,T]
\end{align}
is a fractional Brownian motion on $(\Omega, \mathfrak{A},P)$ with Hurst parameter $H\in (0,1/2)$, where $\frac{dP}{d\widetilde{P}}=\xi_T$. Hence, because of \eqref{VI_weak}, the couple $(X,B^H)$ is a weak solution of \eqref{VI_SDE} on $(\Omega, \mathfrak{A}, P)$. 

Henceforth, we confine ourselves to the filtered probability space $(\Omega, \mathfrak{A}, P)$, $\mathcal{F}=\{\mathcal{F}_t\}_{t\in [0,T]}$ which carries the weak solution $(X,B^H)$ of \eqref{VI_SDE}.

\begin{rem}\label{VI_stochbasisrmk}
As outlined in the scheme above, the main challenge to establish existence of a strong solution is now to show that $X$ is $\mathcal{F}$-adapted. Indeed, in that case $X_t=F_t(B_\cdot^H)$ for some family of measurable functionals $F_t$, $t\in [0,T]$ on $C([0,T];\R^d)$, and for any other stochastic basis $(\hat{\Omega}, \hat{\mathfrak{A}}, \hat{P},\hat{B})$ one gets that $X_t:=F_t(\hat{B}_\cdot)$, $t\in [0,T]$, is a $\hat{B}$-adapted solution to SDE~(\ref{VI_SDE}).  But this means exactly the existence of a strong solution to SDE~(\ref{VI_SDE}).
\end{rem}

\begin{rem}
It is worth to remark that one actually has existence of weak solutions for any $H\in (0,1/2)$ and that weak solutions for bounded $b$ are weakly unique since the estimates from Lemma \ref{VI_novikov} also hold with $X$ in place of $\widetilde{B}^H$. For this reason, the main challenge is to show that when $H$ is small enough such solutions are in fact strong. Then weak uniqueness implies strong uniqueness. See \cite{RY2004}.
\end{rem}

We now turn to the second step of our procedure.

\begin{lem}\label{VI_weakconv}
Let $\{b_n\}_{n\geq 0} \subset C_c^{\infty} ([0,T]\times \R^d)$ be an approximating sequence of $b$ in the sense of \eqref{approxb}. Denote by $X^{n}=\{X_t^n, t\in [0,T]\}$ the corresponding solutions of \eqref{VI_SDE} if we replace $b$ by $b_n$, $n\geq 0$. Then for every $t\in [0,T]$ and bounded continuous function $\varphi:\R^d \rightarrow \R$ we have that
$$\varphi(X_t^{n}) \xrightarrow{n \to \infty} \textrm{\emph{E}}\left[ \varphi(X_t) |\mathcal{F}_t \right],$$
weakly in $L^2(\Omega; \mathcal{F}_t)$.
\end{lem}
\begin{proof}
For a moment let us just, without loss of generality, assume that $x=0$. First we show that
\begin{align} \label{VI_doleansDadeConvergence}
\mathcal{E}\left( \int_0^t K_H^{-1}\left(\int_0^{\cdot} b_n(r,B^H_r)dr\right)^{\ast}(s) dW_s\right)  \rightarrow  \mathcal{E}\left( \int_0^t K_H^{-1}\left(\int_0^{\cdot} b(r,B^H_r)dr\right)^{\ast}(s) dW_s\right)
\end{align}
in $L^p(\Omega)$ for all $p \geq 1$. To see this, note that 
$$
K_H^{-1}\left(\int_0^{\cdot} b_n(r,B^H_r)dr\right)(s) \rightarrow K_H^{-1}\left(\int_0^{\cdot} b(r,B^H_r)dr\right)(s)
$$
in probability for all $s$. Indeed, similar computations as in Lemma \ref{VI_novikov} give
\begin{align*}
\textrm{ E}\Bigg[\Big|& K_H^{-1}\left(\int_0^{\cdot}  b_n(r,B^H_r)dr\right)(s) -   K_H^{-1}\left(\int_0^{\cdot} b(r,B^H_r)dr\right)(s)\Big| \Bigg]
\\
\leq & \, \frac{s^{H- 1/2}}{\Gamma (\frac{1}{2} - H)} \int_0^s (s-r)^{-1/2 -H}r^{1/2 - H} \textrm{ E}[|b_n(r,B^H_r) - b(r,B^H_r)|] dr \\
=& \, \frac{s^{H- 1/2}}{\Gamma (\frac{1}{2} - H)} \int_0^s (s-r)^{-1/2 -H}r^{1/2 - H} \int_{\mathbb{R}^d} |b_n(r,y) - b(r,y)| (2\pi r^{2H})^{-d/2} \exp \left\{-\frac{y^2}{2 r^{2H}} \right\} dy dr  \rightarrow 0
\end{align*}
as $n \rightarrow \infty$ since $b_n(t,x)\rightarrow b(t,x)$ for a.e. $(t,x)$.

Moreover, $\left\{ K_H^{-1}(\int_0^{\cdot} b_n(r,B_r^H)dr) \right\}_{n \geq 0}$ is bounded in $L^2([0,t] \times \Omega; \mathbb{R}^d)$. This is directly seen from (\ref{VI_fracL2}) in Lemma \ref{VI_novikov}.

Consequently
$$
\int_0^t K_H^{-1}\left(\int_0^{\cdot} b_n(r,B_r^H)dr \right)^{\ast}(s) dW_s \rightarrow  \int_0^t K_H^{-1}\left(\int_0^{\cdot} b(r,B_r^H)dr\right)^{\ast}(s) dW_s
$$
and 
$$
\int_0^t \left|K_H^{-1}\left(\int_0^{\cdot} b_n(r,B_r^H)dr\right)(s)\right|^2 ds \rightarrow  \int_0^t \left|K_H^{-1}\left(\int_0^{\cdot} b(r,B_r^H)dr\right)(s)\right|^2 ds
$$
in $L^2(\Omega)$ since the latter is bounded $L^p(\Omega)$ for any $p \geq 1$, see Lemma \ref{VI_novikov}.

Using the estimate $|e^x - e^y| \leq e^{x+y} |x-y|$, H\"{o}lder's inequality and the bounds in Lemma \ref{VI_novikov} it is clear that (\ref{VI_doleansDadeConvergence}) holds.

Similarly, one also shows that 
$$
\exp\left\{ \left\langle \alpha , \int_s^t b_n(r,B_r^H)dr  \right\rangle \right\} \rightarrow \exp \left\{ \left\langle \alpha , \int_s^t b(r,B_r^H)dr \right\rangle   \right\} 
$$
in $L^p(\Omega)$ for all $p \geq 1$, $0 \leq s \leq t \leq T$, $\alpha \in \R^d$.

To conclude the proof we note that the set
$$
\Sigma_t := \left\{ \exp\{ \sum_{j=1}^k \langle \alpha_j, B^H_{t_j} - B^H_{t_{j-1}} \rangle \} : \{ \alpha_j \}_{j=1}^k \subset \mathbb{R}^d , 0= t_0 < \dots < t_k =t , k \geq 1 \right\}
$$
is a total subspace of $L^2(\Omega, \mathcal{F}_t, P)$ and we may thus restrict ourselves to show the convergence
$$
\lim_{n \rightarrow \infty} \textrm{ E}\left[ \left( \varphi(X_t^n) - \textrm{ E}[\varphi(X_t)|\mathcal{F}_t] \right) \xi \right] = 0
$$
for all $\xi \in \Sigma_t$. To this end, we notice that $\varphi$ is of linear growth and hence $\varphi(B_t^H)$ has all moments. Consequently we have the following convergence
$$
\textrm{ E}\left[  \varphi(X_t^n) \exp\left\{ \sum_{j=1}^k \langle \alpha_j , B^H_{t_j} - B^H_{t_{j-1}} \rangle \right\} \right]$$
$$= \textrm{ E}\left[  \varphi(X_t^n) \exp\left\{ \sum_{j=1}^k \langle \alpha_j , X^n_{t_j} - X^n_{t_{j-1}}  - \int_{t_{j-1}}^{t_j} b_n(s,X_s^n) ds \rangle \right\} \right]
$$
$$ 
= \textrm{ E}[  \varphi(B_t^H) \exp\{ \sum_{j=1}^k\langle \alpha_j , B^H_{t_j} - B^H_{t_{j-1}} -  \int_{t_{j-1}}^{t_j} b_n(s,B^H_s) ds \rangle\} \mathcal{E}\left(\int_0^t K_H^{-1} \left( \int_0^{\cdot}b_n(r,B^H_r)dr  \right)^\ast(s) dW_s \right)]
$$
$$
\rightarrow \textrm{ E}[  \varphi(B_t^H) \exp\{ \sum_{j=1}^k \langle \alpha_j , B^H_{t_j} - B^H_{t_{j-1}}- \int_{t_{j-1}}^{t_j} b(s,B^H_s) ds \rangle\} \mathcal{E}\left(\int_0^t K_H^{-1} \left( \int_0^{\cdot}b(r,B^H_r)dr  \right)^\ast(s) dW_s \right)]
$$
$$
=  \textrm{ E}[  \varphi(X_t) \exp\{ \sum_{j=1}^k \langle \alpha_j , B^H_{t_j} - B^H_{t_{j-1}} \rangle \} ]
$$
$$
=  \textrm{ E}[   \textrm{ E}[ \varphi(X_t) | \mathcal{F}_t]  \exp\{ \sum_{j=1}^k \langle \alpha_j ,  B^H_{t_j} - B^H_{t_{j-1}}\rangle \} ] .
$$
\end{proof}

We continue to proving the third step of our scheme. This is the most challenging part. The following result is based on a compactness criterion for subsets of $L^2(\Omega)$ which is summarised in the Appendix.

\begin{lem}\label{VI_relcomp}
Let $\{b_n\}_{n\geq 0} \subset C_c^{\infty} ([0,T]\times \R^d)$ an approximating sequence of $b$ in the sense of \eqref{approxb}. Fix $t\in [0,T]$ and denote by $X_t^{n}$ the corresponding solutions of \eqref{VI_SDE} if we replace $b$ by $b_n$, $n\geq 0$. Then there exists a $\beta\in (0,1/2)$ such that
$$\sup_{n\geq 0} \int_0^t \int_0^t \frac{ \textrm{\emph{E}}[\|D_\theta X_t^{n} - D_{\theta'} X_t^{n}\|^2]}{|\theta' - \theta|^{1+2\beta}} d\theta' d\theta \leq \sup_{n\geq 0}C_{H,d,T}(\|b_n\|_{L_{\infty}^{\infty}}, \|b_n\|_{L_{\infty}^{1}})<\infty$$
and
\begin{align}\label{VI_Mallest}
\sup_{n\geq 0} \|D_{\cdot} X_t^{n}\|_{L^2(\Omega\times [0,T],\R^{d\times d})} \leq \sup_{n\geq 0} C_{H,d,T}(\|b_n\|_{L_{\infty}^{\infty}}, \|b_n\|_{L_{\infty}^{1}})<\infty
\end{align}
for some continuous function $C_{H,d,T}:[0,\infty)^2 \rightarrow [0,\infty)$. Here, $\|\cdot\|$ denotes the maximum norm in $\R^{d\times d}$.
\end{lem}
\begin{proof}
Fix $t\in [0,T]$ and take $\theta,\theta'>0$ such that $0<\theta'<\theta < t$. Using the chain rule for the Malliavin derivative, see \cite[Proposition 1.2.3]{Nua10}, we have
$$D_{\theta} X_t^{n} = K_H(t,\theta) I_{d} + \int_{\theta}^t b_n'(s,X_s^{n}) D_{\theta} X_s^{n} ds,$$
where the above equality is meant in the $L^p$-sense with respect to time. Here, $b_n'(s,z) = \left(\frac{\partial}{\partial z_j} b_n^{(i)} (s,z) \right)_{i,j=1,\dots, d}$ denotes the Jacobian matrix of $b_n$ and $I_d$ the identity matrix in $\R^{d\times d}$. Thus we have

\begin{align*}
D_{\theta'} X_t^{n} -& D_{\theta} X_t^{n} = K_H(t,\theta')I_d- K_H(t,\theta)I_d\\
&+\int_{\theta'}^t b_n'(s,X_s^{n}) D_{\theta'} X_s^{n} ds - \int_{\theta}^t b_n'(s,X_s^{n}) D_{\theta} X_s^{n} ds\\
=& K_H(t,\theta')I_d- K_H(t,\theta) I_d\\
&+\int_{\theta'}^{\theta} b_n'(s,X_s^{n}) D_{\theta'} X_s^{n} ds+\int_{\theta}^t  b_n'(s,X_s^{n}) (D_{\theta'} X_s^{n} - D_{\theta}X_s^{n}) ds\\
=& K_H(t,\theta')I_d - K_H(t,\theta) I_d + D_{\theta'}X_{\theta}^{n} - K_H (\theta,\theta')I_d  \\
&+ \int_{\theta}^t b_n'(s,X_s^{n})(D_{\theta'} X_s^{n} - D_{\theta}X_s^{n})ds.
\end{align*}
Using Picard iteration applied to the above equation we may write

\begin{align*}
D_{\theta'} X_t^{n} -& D_{\theta} X_t^{n} = K_H(t,\theta')I_d - K_H (t,\theta)I_d\\
&+ \sum_{m=1}^{\infty} \int_{\Delta_{\theta,t}^m} \prod_{j=1}^m  b_n'(s_j,X_{s_j}^{n}) \left(K_H(s_m,\theta')I_d - K_H (s_m,\theta)I_d\right) ds_m \cdots ds_1\\
&+ \left( I_d + \sum_{m=1}^{\infty} \int_{\Delta_{\theta,t}^m} \prod_{j=1}^m  b_n'(s_j,X_{s_j}^{n})ds_m \cdots ds_1 \right) \left(D_{\theta'} X_{\theta}^{n} - K_H(\theta,\theta') I_d\right).
\end{align*}
On the other hand, observe that one may again write
\begin{align*}
D_{\theta'} X_{\theta}^{n} - K_H(\theta,\theta')I_d = \sum_{m=1}^{\infty} \int_{\Delta_{\theta',\theta}^m} \prod_{j=1}^m  b_n'(s_j,X_{s_j}^{n}) (K_H(s_m,\theta') I_d) \, ds_m \cdots ds_1.
\end{align*}
Altogether, we can write
$$D_{\theta'} X_t^{n} - D_{\theta} X_t^{n} = I_1(\theta',\theta) + I_2^n (\theta',\theta)+ I_3^n (\theta',\theta),$$
where
\begin{align*}
I_1(\theta',\theta) :=&  K_H(t,\theta')I_d - K_H (t,\theta)I_d\\
I_2^n(\theta',\theta) :=& \sum_{m=1}^{\infty} \int_{\Delta_{\theta,t}^m} \prod_{j=1}^m  b_n'(s_j,X_{s_j}^{n}) \left( K_H(s_m,\theta')I_d - K_H (s_m,\theta)I_d \right) ds_m \cdots ds_1\\
I_3^n(\theta',\theta) :=&\left(I_d+ \sum_{m=1}^{\infty} \int_{\Delta_{\theta,t}^m} \prod_{j=1}^m  b_n'(s_j,X_{s_j}^{n})ds_m \cdots ds_1 \right)\\
&\times\left(\sum_{m=1}^{\infty} \int_{\Delta_{\theta',\theta}^m} \prod_{j=1}^m  b_n'(s_j,X_{s_j}^{n}) (K_H(s_m,\theta')I_d) ds_m \cdots ds_1.\right).
\end{align*}

It follows from Lemma \ref{VI_doubleint} that
\begin{align}\label{VI_double1}
\int_0^t \int_0^t \frac{\|I_1(\theta',\theta)\|_{L^2(\Omega)}^2}{|\theta'-\theta|^{1+2\beta}}d\theta d\theta' = \int_0^t \int_0^t \frac{|K_H(t,\theta')-K_H(t,\theta)|^2}{|\theta'-\theta|^{1+2\beta}}d\theta d\theta'<\infty
\end{align}
for a suitably small $\beta \in (0,1/2)$.

Let us continue with the term $I_2^n(\theta',\theta)$. Then Girsanov's theorem, Cauchy-Schwarz inequality and Lemma \ref{VI_novikov} imply
\begin{align*}
 \textrm{ E}[&\| I_2^n(\theta',\theta) \|^2]\\
&\leq \widetilde{C}(\|b_n\|_{L_{\infty}^{\infty}}) E\left[ \left\|\sum_{m=1}^{\infty} \int_{\Delta_{\theta,t}^m} \prod_{j=1}^m  b_n'(s_j,x+B_{s_j}^H) \left( K_H(s_m,\theta')I_d - K_H (s_m,\theta)I_d \right) ds_m \cdots ds_1\right\|^4 \right]^{1/2},
\end{align*}
where $\widetilde{C}:[0,\infty) \rightarrow [0,\infty)$ is the function from Lemma \ref{VI_novikov}. Taking the supremum over $n$ we have
$$\sup_{n\geq 0}\widetilde{C}(\|b_n\|_{L_{\infty}^{\infty}}) =: C_1 <\infty.$$

Then,

\begin{align*}
 \textrm{ E}[\| I_2^n(\theta',\theta) \|^2] \leq& C_1\Bigg(\sum_{m=1}^{\infty} \sum_{i,j=1}^d \sum_{l_1,\dots, l_{m-1}=1}^d \Bigg\|\int_{\Delta_{\theta,t}^m} \frac{\partial}{\partial x_{l_1}} b_n^{(i)} (s_1,x+B_{s_1}^H) \frac{\partial}{\partial x_{l_2}}b_n^{(l_1)} (s_2,x+B_{s_2}^H) \cdots\\
& \cdots \frac{\partial}{\partial x_j}b_n^{(l_{m-1})} (s_m,x+B_{s_m}^H) \left( K_H(s_m,\theta') - K_H (s_m,\theta) \right) ds_m \cdots ds_1\Bigg\|_{L^4(\Omega, \R)}\Bigg)^{2}  .
\end{align*}

Now look at the expression
\begin{align}\label{VI_I}
J_2^n(\theta',\theta) := \int_{\Delta_{\theta,t}^m} \frac{\partial}{\partial x_{l_1}} b_n^{(i)} (s_1,x+B_{s_1}^H) \cdots \frac{\partial}{\partial x_j}b_n^{(l_{m-1})} (s_m,x+B_{s_m}^H) \left( K_H(s_m,\theta') - K_H (s_m,\theta) \right) ds.
\end{align}
Then, shuffling $J_2^n(\theta',\theta)$ as shown in \eqref{shuffleIntegral}, one can write $(J_2^n(\theta',\theta))^2$ as a sum of at most $2^{2m}$ summands of length $2m$ of the form
\begin{align}\label{VI_II}
\int_{\Delta_{\theta,t}^{2m}} g_1^n (s_1,B_{s_1}^H) \cdots g_{2m}^n (s_{2m},B_{s_{2m}}^H) ds_{2m} \cdots ds_1,
\end{align}
where for each $l=1,\dots, 2m$,
$$g_l^n(\cdot, B_{\cdot}^H) \in \left\{ \frac{\partial}{\partial x_j} b_n^{(i)} (\cdot, x+B_{\cdot}^H),  \frac{\partial}{\partial x_j} b_n^{(i)} (\cdot, x+B_{\cdot}^H)\left( K_H(\cdot,\theta') - K_H (\cdot,\theta) \right), \, i,j=1,\dots,d\right\}.$$

Repeating this argument once again, we find that $J_2^n(\theta',\theta)^4$ can be expressed as a sum of, at most, $2^{8m}$ summands of length $4m$ of the form
\begin{align}\label{VI_III}
\int_{\Delta_{\theta,t}^{4m}} g_1^n (s_1,B_{s_1}^H) \cdots g_{4m}^n (s_{4m},B_{s_{4m}}^H) ds_{4m} \cdots ds_1,
\end{align}
where for each $l=1,\dots, 4m$,
$$g_l^n(\cdot, B_{\cdot}^H) \in \left\{ \frac{\partial}{\partial x_j} b_n^{(i)} (\cdot, x+B_{\cdot}^H),  \frac{\partial}{\partial x_j} b_n^{(i)} (\cdot, x+B_{\cdot}^H)\left( K_H(\cdot,\theta') - K_H (\cdot,\theta) \right), \, i,j=1,\dots,d\right\}.$$

It is important to note that the function $\left( K_H(\cdot,\theta') - K_H (\cdot,\theta) \right)$ appears only once in term \eqref{VI_I} and hence only four times in term \eqref{VI_III}. So there are indices $j_1,\dots, j_4 \in \{1,\dots, 4m\}$ such that we can write \eqref{VI_III} as
$$\int_{\Delta_{\theta,t}^{4m}} \left(\prod_{j=1}^{4m} b_j^n(s_j, B_{s_j}^H)\right)  \prod_{i=1}^4 \left( K_H(s_{j_i},\theta') - K_H (s_{j_i},\theta)\right) ds_{4m} \cdots ds_1,$$
where
$$b_l^n(\cdot, B_{\cdot}^H) \in \left\{ \frac{\partial}{\partial x_j} b_n^{(i)} (\cdot, x+B_{\cdot}^H), \, i,j=1,\dots,d\right\}, \quad l=1,\dots,4m.$$

The latter enables us to use the estimate from Proposition \ref{mainestimate1} with $\sum_{j=1}^{4m} \varepsilon_j=4$, $\sum_{l=1}^{d}\alpha _{\lbrack \sigma (j)]}^{(l)}=1$ for all $j$, $\left\vert \alpha \right\vert =4m$ and Remark \ref{Remark 3.4}. Therefore we get that
\begin{align*}
 \textrm{ E}(J_2^n(\theta',\theta))^4 \leq   \left(\frac{\theta-\theta'}{\theta \theta'}\right)^{4\gamma} \theta^{4\left(H-\frac{1}{2}-\gamma\right)} C^{4m}  \|b_n\|_{L_{\infty}^{1}}^{4m} A_m^{\gamma}(H,d, |t-\theta|) 
\end{align*}
whenever $H<\frac{1}{2(d+2)}$ and $\gamma\in (0,H)$, where
\begin{equation*} 
A_m^{\gamma }(H,d,\left\vert t-\theta \right\vert ):=\frac{((8m)!)^{1/4}(t-%
\theta )^{-H(4m(d+2))-4(H-\frac{1}{2}-\gamma )+4m}}{\Gamma (-H(d+2)8m+8(H-%
\frac{1}{2}-\gamma )+8m)^{1/2}}. \label{Agamma}
\end{equation*}
Altogether, we see that
\begin{align*}
 \textrm{ E}\left[ \|I_2^n (\theta',\theta)\|^2 \right] \leq \left(\frac{\theta-\theta'}{\theta \theta'}\right)^{2\gamma} \theta^{2\left(H-\frac{1}{2}-\gamma\right)} \left(\sum_{m=1}^\infty d^{m+1} C^m \|b_n\|_{L_{\infty}^1}^m A_m^{\gamma}(H,d,|T|)^{1/4}   \right)^2.
\end{align*}

Since $H<\frac{1}{2(d+2)}$, we see that the latter sum is convergent.
Hence, we can find a continuous function $C_{H,d,T}:[0,\infty)^2 \rightarrow [0,\infty)$ such that
\begin{align*}
\sup_{n\geq 0} \textrm{ E}\left[ \|I_2^n (\theta',\theta)\|^2 \right] \leq \sup_{n\geq 0} C_{H,d,T}(\|b_n\|_{L_{\infty}^{\infty}}, \|b_n\|_{L_{\infty}^{1}})\left(\frac{\theta-\theta'}{\theta \theta'}\right)^{2\gamma} \theta^{2\left(H-\frac{1}{2}-\gamma\right)}
\end{align*}
for $\gamma\in (0,H)$ provided that $H<\frac{1}{2(2+d)}$. It is easy to see that we can choose $\gamma \in (0,H)$ such that there is a suitably small $\beta\in (0,1/2)$, $0<\beta< \gamma<H<1/2$ so that it follows from Lemma \ref{VI_doubleint} that
\begin{align}\label{VI_double2}
\int_0^t \int_0^t \left|\frac{\theta-\theta'}{\theta \theta'}\right|^{2\gamma} |\theta|^{2\left(H-\frac{1}{2}-\gamma\right)} |\theta-\theta'|^{-1-2\beta} d\theta' d\theta <\infty,
\end{align}
for every $t\in (0,T]$.

We now turn to the term $I_3^n(\theta',\theta)$. Observe that term $I_3^n (\theta',\theta)$ is the product of two terms, where the first one will simply be bounded uniformly in $\theta,t \in [0,T]$ under expectation. This can be shown by following meticulously the same steps as we did for $I_2^n(\theta',\theta)$ and observing that in virtue of Proposition \ref{mainestimate2} with $\varepsilon_j = 0$ for all $j$ the singularity in $\theta$ vanishes.

Again Girsanov's theorem, Cauchy-Schwarz inequality several times and Lemma \ref{VI_novikov} lead to
\begin{align*}
 \textrm{ E}[\|I_3^n(\theta',\theta)\|^2] \leq& \widehat{C}(\|b_n\|_{L_{\infty}^{\infty}}) \left\|I_d+ \sum_{m=1}^{\infty} \int_{\Delta_{\theta,t}^m} \prod_{j=1}^m  b_n'(s_j,x+B_{s_j}^H)ds_m \cdots ds_1 \right\|_{L^8(\Omega, \R^{d\times d}) }^2 \\
&\times \left\| \sum_{m=1}^{\infty} \int_{\Delta_{\theta',\theta}^m} \prod_{j=1}^m  b_n'(s_j,x+B_{s_j}^H) K_H(s_m,\theta') ds_m \cdots ds_1\right\|_{L^4(\Omega, \R^{d\times d}) }^2,
\end{align*}
where $\widehat{C}:[0,\infty)\rightarrow [0,\infty)$ denotes the corresponding function obtained from Lemma \ref{VI_novikov} which satisfies
$$\sup_{n\geq 0} \widehat{C}(\|b_n\|_{L_{\infty}^{\infty}})=: C_2<\infty.$$

Again, we have
\begin{align*}
 \textrm{ E} [\|I_3^n(\theta',\theta)\|^2] \leq & C_2 \Bigg(1+ \sum_{m=1}^{\infty} \sum_{i,j=1}^d \sum_{l_1,\dots, l_{m-1}=1}^d \Bigg\|\int_{\Delta_{\theta,t}^m} \frac{\partial}{\partial x_{l_1}} b_n^{(i)}(s_1,x+B_{s_1}^H) \cdots\\
&\cdots  \frac{\partial}{\partial x_j} b_n^{(l_{m-1})} (s_m,x+B_{s_m}^H) ds_m \cdots ds_1 \Bigg\|_{L^8(\Omega, \R)}  \Bigg)^2\\
&\times \Bigg( \sum_{m=1}^{\infty}\sum_{i,j=1}^d \sum_{l_1,\dots, l_{m-1}=1}^d \Bigg\|\int_{\Delta_{\theta',\theta}^m} \frac{\partial}{\partial x_{l_1}} b_n^{(i)}(s_1,x+B_{s_1}^H) \cdots \\
&\cdots  \frac{\partial}{\partial x_j} b_n^{(l_{m-1})} (s_m,x+B_{s_m}^H) K_H(s_m,\theta') ds_m \cdots ds_1\Bigg\|_{L^4(\Omega, \R)} \Bigg)^2.
\end{align*}

Using exactly the same reasoning as for $I_2^n(\theta',\theta)$ we see that the first factor can be bounded by some finite constant $C_3 (\|b_n\|_{L_{\infty}^{1,\infty}})$ depending on $H$, $d$, $T$ and $\|b_n\|_{L_{\infty}^{1,\infty}}$, i.e.
\begin{align*}
 \textrm{ E}[\|I_3^n(\theta',\theta)\|^2] \leq& C_3(\|b_n\|_{L_{\infty}^{1,\infty}}) \Bigg( \sum_{m=1}^{\infty}\sum_{i,j=1}^d \sum_{l_1,\dots, l_{m-1}=1}^d \Bigg\|\int_{\Delta_{\theta',\theta}^m} \frac{\partial}{\partial x_{l_1}} b_n^{(i)}(s_1,x+B_{s_1}^H) \cdots\\
&\cdots \frac{\partial}{\partial x_j} b_n^{(l_{m-1})} (s_m,x+B_{s_m}^H) K_H(s_m,\theta') ds_m \cdots ds_1\Bigg\|_{L^4(\Omega, \R)} \Bigg)^2.
\end{align*}
As before, look at
\begin{align}\label{VI_IV}
J_3^n(\theta',\theta) := \int_{\Delta_{\theta',\theta}^m} \frac{\partial}{\partial x_{l_1}} b_n^{(i)} (s_1,x+B_{s_1}^H) \cdots \frac{\partial}{\partial x_j}b_n^{(l_{m-1})} (s_m,x+B_{s_m}^H) K_H(s_m,\theta') ds_m \cdots ds_1.
\end{align}

We can express $(J_3^n(\theta',\theta))^4$ as a sum of, at most, $2^{8m}$ summands of length $4m$ of the form
\begin{align}\label{VI_V}
\int_{\Delta_{\theta',\theta}^{4m}} g_1^n (s_1,B_{s_1}^H) \cdots g_{4m}^n (s_{4m},B_{s_{4m}}^H) ds_{4m} \cdots ds_1,
\end{align}
where for each $l=1,\dots, 4m$,
$$g_l^n(\cdot, B_{\cdot}^H) \in \left\{ \frac{\partial}{\partial x_j} b_n^{(i)} (\cdot, x+B_{\cdot}^H),  \frac{\partial}{\partial x_j} b_n^{(i)} (\cdot, x+B_{\cdot}^H) K_H(\cdot,\theta'), \, i,j=1,\dots,d\right\},$$
where the factor $K_H(\cdot,\theta')$ is repeated four times in the integrand of \eqref{VI_V}. Now we can simply apply Proposition \ref{mainestimate2} with $\sum_{j=1}^{4m}\varepsilon_j=4$, $\sum_{l=1}^{d}\alpha _{\lbrack \sigma (j)]}^{(l)}=1$ for all $j$, $\left\vert \alpha \right\vert =4m$ and Remark \ref{Remark 3.4} and obtain that
$$ \textrm{ E}[(J_3^n(\theta',\theta))^4] \leq \theta^{4\left(H-\frac{1}{2}\right)} C^{4m} \|b_n\|_{L_{\infty}^{1}}^{4m} A_m^{0}(H,d, |\theta-\theta'|),$$
whenever $H<\frac{1}{2(2+d)}$ where $A_m^{0}(H,d, |\theta-\theta'|)$ is defined as in \eqref{Agamma} by inserting $\gamma =0$.

As a result,
$$ \textrm{ E}[\|I_3^n(\theta',\theta)\|^2] \leq \theta^{2\left(H-\frac{1}{2}\right)}\left(\sum_{m=1}^\infty d^{m+1} C^m  \|b_n\|_{L_{\infty}^{1}}^{m} A_m^0(H,d,|\theta-\theta'|)^{1/4}\right)^2.$$
Note that the above series converges, because $H<\frac{1}{2(2+d)}$.

Since the exponent of $|\theta-\theta'|$ appearing in $A_m^0(b_n,H,d,|\theta-\theta'|)$ is strictly positive by assumption, we can find a small enough $\varepsilon>0$ and a continuous function $C_{H,d,T}:[0,\infty)^2 \rightarrow [0,\infty)$ such that
$$\sup_{n\geq 0} \textrm{ E}[\|I_3^n(\theta',\theta)\|^2] \leq \sup_{n\geq 0} C_{H,d,T}(\|b_n\|_{L_{\infty}^{\infty}},\|b_n\|_{L_{\infty}^{1}}) |\theta|^{2\left(H-\frac{1}{2}\right)}|\theta - \theta'|^{\varepsilon}$$
provided $H<\frac{1}{2(2+d)}$. Then again, it is easy to see that we can choose $\beta\in(0,1/2)$ small enough so that it follows from Lemma \ref{VI_doubleint} that
\begin{align}\label{VI_double3}
\int_0^t \int_0^t |\theta|^{2\left(H-\frac{1}{2}\right)}|\theta - \theta'|^{\varepsilon -1 -2\beta} d\theta' d\theta <\infty,
\end{align}
for every $t\in [0,T]$.

Altogether, taking a suitable $\beta$ so that \eqref{VI_double1}, \eqref{VI_double2} and \eqref{VI_double3} are finite, we have
$$\sup_{n\geq 0} \int_0^t \int_0^t \frac{ \textrm{ E}[\| D_{\theta'} X_t^n - D_{\theta} X_t^n \|^2]}{|\theta' - \theta|^{1+2\beta}} d\theta' d\theta \leq \sup_{n \geq 0} C_{H,d,T}(\|b_n\|_{L_{\infty}^{\infty}},\|b_n\|_{L_{\infty}^{1}}) < \infty$$
for some continuous function $C_{H,d,T}:[0,\infty)^2 \rightarrow [0,\infty)$.

Similar computations show that
$$\sup_{n\geq 0} \|D_{\cdot} X_t^n \|_{L^2(\Omega\times [0,T],\R^{d\times d})} \leq \sup_{n \geq 0} C_{H,d,T}(\|b_n\|_{L_{\infty}^{\infty}}, \|b_n\|_{L_{\infty}^{1}}) < \infty.$$
\end{proof}

\begin{cor}\label{VI_L2conv}
Let $\{b_n\}_{n\geq 0} \subset C_c^{\infty} ([0,T]\times \R^d)$ the approximating sequence of $b$ in the sense of \eqref{VI_Xn}. Denote by $X_t^{n}$ the corresponding solutions of \eqref{approxb} if we replace $b$ by $b_n$, $n\geq 0$. Then for every $t\in [0,T]$ and bounded continuous function $\varphi:\R^d \rightarrow \R$ we have
$$\varphi(X_t^{n}) \xrightarrow{n \to \infty} \varphi( \textrm{\emph{E}}\left[ X_t |\mathcal{F}_t \right])$$
strongly in $L^2(\Omega; \mathcal{F}_t)$. In addition, $ \textrm{\emph{E}}\left[X_t|\mathcal{F}_t\right]$ is Malliavin differentiable for every $t\in [0,T]$.
\end{cor}
\begin{proof}
This is an immediate consequence of the relative compactness from Corollary \ref{VI_compactcrit} in connection with Lemma \ref{VI_relcomp} and because of Lemma \ref{VI_weakconv} we can identify the limit as being $ \textrm{E}[X_t|\mathcal{F}_t]$, then the convergence holds for any bounded continuous function as well. The Malliavin differentiability of $\textrm{E}[X_t|\mathcal{F}_t]$ is shown by taking $\varphi =I_d$ and estimate \eqref{VI_Mallest} together with \cite[Proposition 1.2.3]{Nua10}.
\end{proof}

Finally, we can prove the main result of this section.

\begin{proof}[Proof of Theorem \ref{VI_mainthm}]
It remains to prove that $X_t$ is $\mathcal{F}_t$-measurable for every $t\in [0,T]$ and by Remark \ref{VI_stochbasisrmk} it then follows that there exists a strong solution in the usual sense that is Malliavin differentiable. Indeed, let $\varphi$ be a globally Lipschitz continuous function, then by Corollary \ref{VI_L2conv} we have, for a subsequence $n_k$, $k\geq 0$, that
$$\varphi(X_t^{n_k}) \rightarrow \varphi(\textrm{E}[X_t|\mathcal{F}_t]), \ \ P-a.s.$$
as $k\to \infty$.

On the other hand, by Lemma \ref{VI_weakconv} we also have
$$\varphi (X_t^{n}) \rightarrow \textrm{E}\left[ \varphi(X_t)|\mathcal{F}_t\right]$$
weakly in $L^2(\Omega;\mathcal{F}_t)$. By the uniqueness of the limit we immediately have
$$\varphi\left( \textrm{E}[X_t|\mathcal{F}_t] \right) =\textrm{E}\left[ \varphi(X_t)|\mathcal{F}_t\right], \ \ P-a.s.$$
which implies that $X_t$ is $\mathcal{F}_t$-measurable for every $t\in [0,T]$.

Finally, to show uniqueness it is enough to show that two given strong solutions are weakly unique, indeed, one can follow the same argument as in \cite[Chapter IX, Exercise (1.20)]{RY2004} which asserts that strong existence and uniqueness in law imply pathwise uniqueness. The argument does not rely on the process being a semimartingale. Since our solutions are, by construction, strong and uniqueness in law follows from Novikov's condition from Lemma \ref{VI_novikov} replacing $B^H$ by $X$ then pathwise uniqueness follows.
\end{proof}

\section{Stochastic flows and regularity properties}\label{Flow}

Henceforward, we will denote by $X_t^{s,x}$ the solution to the following SDE driven by a fractional Brownian motion with $H<1/2$
\begin{align}\label{VI_SDE2}
dX_t^{s,x} =b(t,X_t^{s,x}) dt + dB_t^H, \quad s,t\in [0,T], \quad s\leq t, \quad X_s^{s,x} = x\in \R^d.
\end{align}

We will then assume the hypotheses from Theorem \ref{VI_mainthm} on $b$ and $H$, that is $b\in L_{\infty,\infty}^{1,\infty}$ and $H<\frac{1}{2(d+2)}$. The next result tells us that if $H=H(k)$ is small enough we may gain regularity on $x\mapsto X_t^{s,x}$. In particular, it shows that the strong solution constructed in the former section, in addition to being Malliavin differentiable, is also once weakly differentiable with respect to $x$ since $k=1$. See the authors in \cite{CHOP}, who treated the case $k=2$. 

\begin{thm}\label{VI_derivative}
Let $b \in C_c^{\infty}([0,T]\times \R^{d})$. Fix integers $p\geq 2$ and $k\geq 1$. Then, if $H<\frac{1}{(d-1+2k)}$ we have
$$\sup_{s,t\in [0,T]} \sup_{x\in \R^d} \textrm{\emph{E}}\left[ \left\| \frac{\partial^{k}}{\partial x^{k}} X_t^{s,x}\right\|^{p} \right] \leq C_{k,d,H,p,T} (\|b\|_{L_{\infty}^{\infty}},\|b\|_{L_{\infty}^{1}}),$$
where $C_{k,d,H,p,T}:[0,\infty)^2\rightarrow [0,\infty)$ is a continuous function, depending on $k,d,H,p$ and $T$.
\end{thm}

\begin{proof}
For notational convenience, let us assume that $s=0$ and denote the
corresponding solution by $X_{t}^{x},$ $0\leq t\leq T$ with respect to the
vector field $b\in C_{c}^{\infty }((0,T)\times \mathbb{R}^{d})$. Since the
stochastic flow associated with the smooth vector field $b$ is smooth, too
(compare to e.g. \cite{Kunita}), we find that%
\begin{equation}
\frac{\partial }{\partial x}X_{t}^{x}=I_{d\times
d}+\int_{s}^{t}Db(u,X_{u}^{x})\cdot \frac{\partial }{\partial x}X_{u}^{x}du,
\end{equation}%
where $Db(u,\cdot):\mathbb{R}^{d}\longrightarrow L(\mathbb{R}^{d},\mathbb{R}^{d})$ is
the derivative of $b$ with respect to the space variable.

By employing Picard iteration, we obtain that%
\begin{equation}
\frac{\partial }{\partial x}X_{t}^{x}=I_{d\times d}+\sum_{m\geq
1}\int_{\Delta
_{0,t}^{m}}Db(u,X_{u_{1}}^{x})...Db(u,X_{u_{m}}^{x})du_{m}...du_{1},
\label{FirstOrder}
\end{equation}%
where%
\begin{equation*}
\Delta _{s,t}^{m}=\{(u_{m},...u_{1})\in \lbrack 0,T]^{m}:\theta
<u_{m}<...<u_{1}<t\}.
\end{equation*}

Using dominated convergence, we can differentiate both sides with respect to 
$x$ and see that%
\begin{equation*}
\frac{\partial ^{2}}{\partial x^{2}}X_{t}^{x}=\sum_{m\geq 1}\int_{\Delta
_{0,t}^{m}}\frac{\partial }{\partial x}%
[Db(u,X_{u_{1}}^{x})...Db(u,X_{u_{m}}^{x})]du_{m}...du_{1}.
\end{equation*}%
On the other hand the Leibniz and chain rule give%
\begin{eqnarray*}
&&\frac{\partial }{\partial x}[Db(u_{1},X_{u_{1}}^{x})...Db(u_m,X_{u_{m}}^{x})] \\
&=&\sum_{r=1}^{m}Db(u_1,X_{u_{1}}^{x})...D^{2}b(u_r,X_{u_{r}}^{x})\frac{\partial }{%
\partial x}X_{u_{r}}^{x}...Db(u_m,X_{u_{m}}^{x}),
\end{eqnarray*}%
where $D^{2}b(u,\cdot)=D(Db(u,\cdot)):\mathbb{R}^{d}\longrightarrow L(\mathbb{R}^{d},L(\mathbb{%
R}^{d},\mathbb{R}^{d}))$.

So (\ref{FirstOrder}) implies that%
\begin{eqnarray}
\frac{\partial ^{2}}{\partial x^{2}}X_{t}^{x} &=&\sum_{m_{1}\geq
1}\int_{\Delta
_{0,t}^{m_{1}}}\sum_{r=1}^{m_{1}}Db(u_1,X_{u_{1}}^{x})...D^{2}b(u_r,X_{u_{r}}^{x}) 
\notag \\
&&\times \left( I_{d\times d}+\sum_{m_{2}\geq 1}\int_{\Delta
_{0,u_{r}}^{m_{2}}}Db(v_1,X_{v_{1}}^{x})...Db(v_{m_2},X_{v_{m_{2}}}^{x})dv_{m_{2}}...dv_{1}\right) 
\notag \\
&&\times Db(u_{r+1},X_{u_{r+1}}^{x})...Db(u_{m_1},X_{u_{m_{1}}}^{x})du_{m_{1}}...du_{1} 
\notag \\
&=&\sum_{m_{1}\geq 1}\sum_{r=1}^{m_{1}}\int_{\Delta
_{0,t}^{m_{1}}}Db(u_1, X_{u_{1}}^{x})...D^{2}b(u_r,X_{u_{r}}^{x})...Db(u_{m_1},X_{u_{m_{1}}}^{x})du_{m_{1}}...du_{1}
\notag \\
&&+\sum_{m_{1}\geq 1}\sum_{r=1}^{m_{1}}\sum_{m_{2}\geq 1}\int_{\Delta
_{0,t}^{m_{1}}}\int_{\Delta
_{0,u_{r}}^{m_{2}}}Db(u_1,X_{u_{1}}^{x})...D^{2}b(u_r,X_{u_{r}}^{x})  \notag \\
&&\times
Db(v_1, X_{v_{1}}^{x})...Db(v_{m_2}X_{v_{m_{2}}}^{x})Db(u_{r+1},X_{u_{r+1}}^{x})...Db(u_{m_1},X_{u_{m_{1}}}^{x})
\notag \\
&&dv_{m_{2}}...dv_{1}du_{m_{1}}...du_{1}  \notag \\
&=&:I_{1}+I_{2}.  \label{SecondOrder}
\end{eqnarray}

Next we apply Lemma \ref{OrderDerivatives} (in connection with Lemma \ref%
{partialshuffle}) to the term $I_{2}$ in (\ref{SecondOrder}) and obtain that%
\begin{equation}
I_{2}=\sum_{m_{1}\geq 1}\sum_{r=1}^{m_{1}}\sum_{m_{2}\geq 1}\int_{\Delta
_{0,t}^{m_{1}+m_{2}}}\mathcal{H}%
_{m_{1}+m_{2}}^{X}(u)du_{m_{1}+m_{2}}...du_{1}  \label{l2}
\end{equation}%
for $u=(u_{1},...,u_{m_{1}+m_{2}}),$ where the integrand $\mathcal{H}%
_{m_{1}+m_{2}}^{X}(u)\in \mathbb{R}^{d}\otimes \mathbb{R}^{d}\otimes \mathbb{%
R}^{d}$ has entries given by sums of at most $C(d)^{m_{1}+m_{2}}$ terms,
which are products of length $m_{1}+m_{2}$ of functions belonging to the set%
\begin{equation*}
\left\{ \frac{\partial ^{\gamma ^{(1)}+...+\gamma ^{(d)}}}{\partial ^{\gamma
^{(1)}}x_{1}...\partial ^{\gamma ^{(d)}}x_{d}}b^{(r)}(u,X_{u}^{x}),\text{ }%
r=1,...,d,\text{ }\gamma ^{(1)}+...+\gamma ^{(d)}\leq 2,\text{ }\gamma
^{(l)}\in \mathbb{N}_{0},\text{ }l=1,...,d\right\} .
\end{equation*}%
Here it is crucial to note that second order derivatives of functions in
those products of functions on $\Delta _{0,t}^{m_{1}+m_{2}}$ in (\ref{l2})
only appear once. So the total order of derivatives $\left\vert \alpha
\right\vert $ of those products of functions in connection with Lemma \ref%
{OrderDerivatives} in the Appendix is%
\begin{equation}
\left\vert \alpha \right\vert =m_{1}+m_{2}+1.
\end{equation}%
We now choose $p,c,r\in \lbrack 1,\infty )$ such that $cp=2^{q}$ for some
integer $q$ and $\frac{1}{r}+\frac{1}{c}=1.$ Then we can use H\"{o}lder's
inequality and Girsanov's theorem (see Theorem \ref{VI_girsanov}) in
combination with Lemma \ref{VI_novikov} and find that%
\begin{eqnarray}
&&E[\left\Vert I_{2}\right\Vert ^{p}]  \notag \\
&\leq &C(\left\Vert b\right\Vert _{L_{\infty }^{\infty }})\left(
\sum_{m_{1}\geq 1}\sum_{r=1}^{m_{1}}\sum_{m_{2}\geq 1}\sum_{i\in
I}\left\Vert \int_{\Delta _{0,t}^{m_{1}+m_{2}}}\mathcal{H}%
_{i}^{B^{H}}(u)du_{m_{1}+m_{2}}...du_{1}\right\Vert _{L^{2^{q}}(\Omega ;%
\mathbb{R})}\right) ^{p},  \label{Lp}
\end{eqnarray}%
where $C:[0,\infty )\longrightarrow \lbrack 0,\infty )$ is a continuous
function depending on $p$. Here $\#I\leq K^{m_{1}+m_{2}}$ for a constant $%
K=K(d)$ and the integrands $\mathcal{H}_{i}^{B^{H}}(u)$ take the form 
\begin{equation*}
\mathcal{H}_{i}^{B^{H}}(u)=\prod_{l=1}^{m_{1}+m_{2}}h_{l}(u_{l}),h_{l}\in \Lambda ,l=1,...,m_{1}+m_{2}
\end{equation*}%
where 
\begin{equation*}
\Lambda :=\left\{ 
\begin{array}{c}
\frac{\partial ^{\gamma ^{(1)}+...+\gamma ^{(d)}}}{\partial ^{\gamma
^{(1)}}x_{1}...\partial ^{\gamma ^{(d)}}x_{d}}b^{(r)}(u,x+B_{u}^{H}),\text{ }%
r=1,...,d, \\ 
\gamma ^{(1)}+...+\gamma ^{(d)}\leq 2,\text{ }\gamma ^{(l)}\in \mathbb{N}%
_{0},\text{ }l=1,...,d%
\end{array}%
\right\} .
\end{equation*}%
As before here functions with second order derivatives only appear once in
those products.

Let 
\begin{equation*}
J=\left( \int_{\Delta _{0,t}^{m_{1}+m_{2}}}\mathcal{H}%
_{i}^{B^{H}}(u)du_{m_{1}+m_{2}}...du_{1}\right) ^{2^{q}}.
\end{equation*}%
By employing Lemma \ref{partialshuffle} once more in the Appendix,
successively, we find that $J$ can be represented as a sum of, at most of
length $K(q)^{m_{1}+m_{2}}$ with summands of the form%
\begin{equation}
\int_{\Delta
_{0,t}^{2^{q}(m_{1}+m_{2})}}\prod_{l=1}^{2^{q}(m_{1}+m_{2})}f_{l}(u_{l})du_{2^{q}(m_{1}+m_{2})}...du_{1},
\label{f}
\end{equation}%
where $f_{l}\in \Lambda $ for all $l$.

Here the number of factors $f_{l}$ in the above product, which have a second
order derivative, is exactly $2^{q}.$ So the total order of the derivatives
involved in (\ref{f}) in connection with Lemma \ref{OrderDerivatives} (where
one in that lemma formally replaces $X_{u}^{x}$ by $x+B_{u}^{H}$ in the
corresponding terms) is 
\begin{equation}
\left\vert \alpha \right\vert =2^{q}(m_{1}+m_{2}+1).  \label{alpha2}
\end{equation}

We now want to apply Theorem \ref{mainestimate2} for $m=2^{q}(m_{1}+m_{2})$
and $\varepsilon _{j}=0$ and see that%
\begin{eqnarray*}
&&\left\vert E\left[ \int_{\Delta
_{0,t}^{2^{q}(m_{1}+m_{2})}}\prod_{l=1}^{2^{q}(m_{1}+m_{2})}f_{l}(u_{l})du_{2^{q}(m_{1}+m_{2})}...du_{1}%
\right] \right\vert  \\
&\leq &C^{m_{1}+m_{2}}(\left\Vert b\right\Vert_{L_{\infty}^1})^{2^{q}(m_{1}+m_{2})} \\
&&\times \frac{((2(2^{q}(m_{1}+m_{2}+1))!)^{1/4}}{\Gamma
(-H(2d2^{q}(m_{1}+m_{2})+42^{q}(m_{1}+m_{2}+1))+22^{q}(m_{1}+m_{2}))^{1/2}}
\end{eqnarray*}%
for a constant $C$ depending on $H,T,d$ and $q$.

Hence the latter in combination with (\ref{Lp}) yields that%
\begin{eqnarray*}
&&E[\left\Vert I_{2}\right\Vert ^{p}] \\
&\leq &C(\left\Vert b\right\Vert _{L_{\infty }^{\infty }})\left(
\sum_{m_{1}\geq 1}\sum_{m_{2}\geq 1}K^{m_{1}+m_{2}}(\left\Vert b\right\Vert
_{L_{\infty}^1})^{2^{q}(m_{1}+m_{2})}\right.  \\
&&\left. \times \frac{((2(2^{q}(m_{1}+m_{2}+1))!)^{1/4}}{\Gamma
(-H(2d2^{q}(m_{1}+m_{2})+42^{q}(m_{1}+m_{2}+1))+22^{q}(m_{1}+m_{2}))^{1/2}}%
)^{1/2^{q}}\right) ^{p}
\end{eqnarray*}%
for a constant $K$ depending on $H,$ $T,$ $d,$ $p$ and $q$.

Since $\frac{1}{2(d+3)}\leq \frac{1}{2(d+2\frac{m_{1}+m_{2}+1}{m_{1}+m_{2}})}
$ for $m_{1},$ $m_{2}\geq 1$, it follows that the above sum converges,
whenever $H<\frac{1}{2(d+3)}$.

On the other hand one obtains by using the same reasoning as before a
similar estimate for $E[\left\Vert I_{1}\right\Vert ^{p}]$. Altogether the
proof follows for $k=2$.

Let us now explain the generalization of the previous line of reasoning to
the case $k\geq 2$: In this case, we get that%
\begin{equation}
\frac{\partial ^{k}}{\partial x^{k}}X_{t}^{x}=I_{1}+...+I_{2^{k-1}},
\label{Ik}
\end{equation}%
where each $I_{i},$ $i=1,...,2^{k-1}$ is a sum of iterated integrals over
simplices of the form $\Delta _{0,u}^{m_{j}},$ $0<u<t,$ $j=1,...,k$ with
integrands, which have at most one product factor $D^{k}b$, whereas the
other factors are of the form $D^{j}b,j\leq k-1$.

In what follows we need the following notation: For given multi-indices $%
m.=(m_{1},...,m_{k})$ and $r:=(r_{1},...,r_{k-1})$ we define%
\begin{equation*}
m_{j}^{-}:=\sum_{i=1}^{j}m_{i}\text{ }
\end{equation*}%
and%
\begin{equation*}
\sum_{\substack{ m\geq 1 \\ r_{l}\leq m_{l}^{-} \\ l=1,...,k-1}}%
:=\sum_{m_{1}\geq 1}\sum_{r_{1}=1}^{m_{1}}\sum_{m_{2}\geq
1}\sum_{r_{2}=1}^{m_{2}^{-}}...\sum_{r_{k-1}=1}^{m_{k-1}^{-}}\sum_{m_{k}\geq
1}.
\end{equation*}%
In the sequel, we restrict ourselves without loss of generality to the
estimation of the summand $I_{2^{k-1}}$ in (\ref{Ik}). In the same way as in
the case $k=2,$ we get by invoking Lemma \ref{OrderDerivatives} (in
connection with Lemma \ref{partialshuffle}) that 
\begin{equation}
I_{2^{k-1}}=\sum_{\substack{ m\geq 1 \\ r_{l}\leq m_{l}^{-} \\ l=1,...,k-1}}%
\int_{\Delta _{0,t}^{m_{1}+...+m_{k}}}\mathcal{H}%
_{m_{1}+...+m_{k}}^{X}(u)du_{m_{1}+m_{2}}...du_{1}
\end{equation}%
for $u=(u_{m_{1}+...+m_{k}},...,u_{1}),$ where the integrand $\mathcal{H}%
_{m_{1}+...+m_{k}}^{X}(u)\in \otimes _{j=1}^{k+1}\mathbb{R}^{d}$ has entries
given by sums of at most $C(d)^{m_{1}+...+m_{k}}$ terms, which are products
of length $m_{1}+...m_{k}$ of functions, which are elements in%
\begin{equation*}
\left\{ 
\begin{array}{c}
\frac{\partial ^{\gamma ^{(1)}+...+\gamma ^{(d)}}}{\partial ^{\gamma
^{(1)}}x_{1}...\partial ^{\gamma ^{(d)}}x_{d}}b^{(r)}(u,X_{u}^{x}),r=1,...,d,
\\ 
\gamma ^{(1)}+...+\gamma ^{(d)}\leq k,\gamma ^{(l)}\in \mathbb{N}%
_{0},l=1,...,d%
\end{array}%
\right\} .
\end{equation*}%
Just as in the case $k=2$ we can employ Lemma \ref{OrderDerivatives} in the
Appendix and obtain that the total order of derivatives $\left\vert \alpha
\right\vert $ of those products of functions is 
\begin{equation}
\left\vert \alpha \right\vert =m_{1}+...+m_{k}+k-1.
\end{equation}%
Then we follow the line of reasoning as before and choose $p,c,r\in \lbrack
1,\infty )$ such that $cp=2^{q}$ for some integer $q$ and $\frac{1}{r}+\frac{%
1}{c}=1$ and get by making use of H\"{o}lder's inequality and Girsanov's
theorem (see Theorem \ref{VI_girsanov}) in connection with Lemma \ref%
{VI_novikov} that%
\begin{eqnarray}
&&E[\left\Vert I_{2^{k-1}}\right\Vert ^{p}]  \notag \\
&\leq &C(\left\Vert b\right\Vert _{L_{\infty }^{\infty }})\left( \sum
_{\substack{ m\geq 1 \\ r_{l}\leq m_{l}^{-} \\ l=1,...,k-1}}\sum_{i\in
I}\left\Vert \int_{\Delta _{0,t}^{m_{1}+m_{2}}}\mathcal{H}%
_{i}^{B^{H}}(u)du_{m_{1}+...+m_{k}}...du_{1}\right\Vert _{L^{2^{q}}(\Omega ;%
\mathbb{R})}\right) ^{p},  \label{Lp2}
\end{eqnarray}%
where $C:[0,\infty )\longrightarrow \lbrack 0,\infty )$ is a continuous
function depending on $p$. Here $\#I\leq K^{m_{1}+...+m_{k}}$ for a constant 
$K=K(d)$ and the integrands $\mathcal{H}_{i}^{B^{H}}(u)$ are of the form 
\begin{equation*}
\mathcal{H}_{i}^{B^{H}}(u)=\prod_{l=1}^{m_{1}+...+m_{k}}h_{l}(u_{l}),%
\text{ }h_{l}\in \Lambda ,\text{ }l=1,...,m_{1}+...+m_{k},
\end{equation*}%
where 
\begin{equation*}
\Lambda :=\left\{ 
\begin{array}{c}
\frac{\partial ^{\gamma ^{(1)}+...+\gamma ^{(d)}}}{\partial ^{\gamma
^{(1)}}x_{1}...\partial ^{\gamma ^{(d)}}x_{d}}b^{(r)}(u,x+B_{u}^{H}),\text{ }%
r=1,...,d, \\ 
\gamma ^{(1)}+...+\gamma ^{(d)}\leq k,\text{ }\gamma ^{(l)}\in \mathbb{N}%
_{0},\text{ }l=1,...,d%
\end{array}%
\right\} .
\end{equation*}%
Define 
\begin{equation*}
J=\left( \int_{\Delta _{0,t}^{m_{1}+...+m_{k}}}\mathcal{H}%
_{i}^{B^{H}}(u)du_{m_{1}+...+m_{k}}...du_{1}\right) ^{2^{q}}.
\end{equation*}%
Once more, repeated application of Lemma \ref{partialshuffle} in the
Appendix shows that $J$ can be written as a sum of, at most of length $%
K(q)^{m_{1}+....m_{k}}$ with summands of the form%
\begin{equation}
\int_{\Delta
_{0,t}^{2^{q}(m_{1}+...+m_{k})}}\prod_{l=1}^{2^{q}(m_{1}+...+m_{k})}f_{l}(u_{l})du_{2^{q}(m_{1}+....+m_{k})}...du_{1},
\label{f2}
\end{equation}%
where $f_{l}\in \Lambda $ for all $l$.

By using Lemma \ref{OrderDerivatives} again (where one in that Lemma
formally replaces $X_{u}^{x}$ by $x+B_{u}^{H}$ in the corresponding
expressions) we find that the total order of the derivatives in the products
of functions in (\ref{f2}) is given by%
\begin{equation}
\left\vert \alpha \right\vert =2^{q}(m_{1}+...+m_{k}+k-1).
\end{equation}

Then Proposition \ref{mainestimate2} for $m=2^{q}(m_{1}+...+m_{k})$ and $%
\varepsilon _{j}=0$ implies that%
\begin{eqnarray*}
&&\left\vert E\left[ \int_{\Delta
_{0,t}^{2^{q}(m_{1}+...+m_{k})}}\prod_{l=1}^{2^{q}(m_{1}+...+m_{k})}f_{l}(u_{l})du_{2^{q}(m_{1}+...+m_{k})}...du_{1}%
\right] \right\vert  \\
&\leq &C^{m_{1}+...+m_{k}}(\left\Vert b\right\Vert _{L^{1}(\mathbb{R}%
^{d})})^{2^{q}(m_{1}+...+m_{k})} \\
&&\times \frac{((2(2^{q}(m_{1}+...+m_{k}+k-1))!)^{1/4}}{\Gamma
(-H(2d2^{q}(m_{1}+...+m_{k})+42^{q}(m_{1}+...+m_{k}+k-1))+22^{q}(m_{1}+...+m_{k}))^{1/2}%
}
\end{eqnarray*}%
for a constant $C$ depending on $H,$ $T,$ $d$ and $q$.

So it follows from (\ref{Lp2}) that%
\begin{eqnarray*}
&&E[\left\Vert I_{2^{k-1}}\right\Vert ^{p}] \\
&\leq &C(\left\Vert b\right\Vert _{L_{\infty }^{\infty }})\left(
\sum_{m_{1}\geq 1}...\sum_{m_{k}\geq 1}K^{m_{1}+...+m_{k}}(\left\Vert
b\right\Vert _{L_{\infty}^1})^{2^{q}(m_{1}+...+m_{k})}\right.  \\
&&\left. \times \frac{((2(2^{q}(m_{1}+...+m_{k}+k-1))!)^{1/4}}{\Gamma
(-H(2d2^{q}(m_{1}+...+m_{k})+42^{q}(m_{1}+...+m_{k}+k-1))+22^{q}(m_{1}+...+m_{k}))^{1/2}%
})^{1/2^{q}}\right) ^{p} \\
&\leq &C(\left\Vert b\right\Vert _{L_{\infty }^{\infty }})\left( \sum_{m\geq
1}\sum_{\substack{ l_{1},...,l_{k}\geq 0: \\ l_{1}+...+l_{k}=m}}%
K^{m}(\left\Vert b\right\Vert _{L_{\infty}^1})^{2^{q}m}\right.  \\
&&\left. \times \frac{((2(2^{q}(m+k-1))!)^{1/4}}{\Gamma
(-H(2d2^{q}m+42^{q}(m+k-1))+22^{q}m)^{1/2}})^{1/2^{q}}\right) ^{p}
\end{eqnarray*}%
for a constant $K$ depending on $H,$ $T,$ $d,$ $p$ and $q$.

Since we required that $H<$ $\frac{1}{2(d-1+2k)}$ the above sum converges.
So the proof follows.
\end{proof}

The following is the main result of this section and shows that the fractional Brownian motion $B^H$ creates a regularising effect on the solution as a function of the initial condition.

\begin{thm}
Assume $b\in L_{\infty,\infty}^{1,\infty}$. Let $U\subset \R^d$ and open and bounded subset and $X=\{X_t,t\in [0,T]\}$ the solution of \eqref{VI_SDE}. Then for a small enough Hurst parameter $H$, that is $H<$ $\frac{1}{2(d-1+2k)}$ it follows that
$$X_t^{\cdot} \in \bigcap_{p>1} L^2(\Omega, W^{k,p}(U)).$$
\end{thm}
\begin{proof}
First of all, approximate the irregular drift vector field $b$ by a sequence of functions $b_n:[0,T]\times \R^d\rightarrow \R^d$, $n\geq 0$ in $C_c^{\infty}([0,T]\times \R^d, \R^d)$ in the sense of \eqref{approxb}. Denote by $X^{n,x} = \{X_t^{n,x}, t\in [0,T]\}$, the corresponding solution to \eqref{VI_SDE} starting from $x\in \R^d$ when $b$ is replaced by $b_n$.

Observe that for any test function $\varphi \in C_0^{\infty}(U, \R^d)$ and fixed $t\in [0,T]$ the set of random variables
$$\langle X_t^{n,\cdot}, \varphi\rangle := \int_{U} \langle X_t^{n,x}, \varphi(x) \rangle_{\R^d} dx, \quad n\geq 0$$
is relatively compact in $L^2(\Omega)$. To show this, we use the compactness criterion from Appendix, in Corollary \ref{VI_compactcrit} in terms of the Malliavin derivative. Since the Malliavin derivative is a closed linear operator we have

\begin{align*}
\textrm{E}\left[\int_0^T |D_{\theta}^j\langle X_t^{n,\cdot}, \varphi\rangle |^2 d\theta \right] \leq  d\|\varphi\|_{L^2(\R^d,\R^d)}^2 \lambda\{\mbox{supp } (\varphi)\} \sup_{x\in U} \textrm{E}\left[\int_0^T \|D_{\theta} X_t^{n,x}\|^2 d\theta\right],
\end{align*}
where $D^j$ denotes the Malliavin derivative in the direction of $W^{(j)}$, $\lambda$ the Lebesgue measure on $\R^d$, $\mbox{supp }(\varphi)$ the support of $\varphi$ and $\|\cdot\|$ a matrix norm. Then taking the sum over all $j=1,\dots,d$ and using Lemma \ref{VI_relcomp} we obtain
\begin{align*}
\sup_{n\geq 0}\|D_{\cdot}\langle X_t^{n,\cdot}, \varphi\rangle \|_{L^2(\Omega\times [0,T])}^2 \leq  C \|\varphi\|_{L^2(\R^d,\R^d)}^2  \lambda\{\mbox{supp } (\varphi)\}.
\end{align*}
In a similar manner we have
$$\sup_{n\geq 0} \int_0^T \int_0^T \frac{\textrm{E}[ \| D_{\theta'} \langle X_t^{n,\cdot}, \varphi \rangle - D_{\theta} \langle X_t^{n,\cdot}, \varphi \rangle \|^2 ]}{|\theta'-\theta|^{1+2\beta}} d\theta d\theta' <\infty$$
for some $\beta\in(0,1/2)$. Hence $\langle X_t^{n,\cdot}, \varphi\rangle$, $n\geq 0$ is relatively compact in $L^2(\Omega)$. Let us denote by $Y_t(\varphi)$ its limit after taking (if necessary) a subsequence.

Following exactly the same reasoning as in Lemma \ref{VI_weakconv} one can show that
$$\langle X_t^{n,\cdot}, \varphi\rangle \xrightarrow{n \to \infty} \langle X_t^{\cdot}, \varphi\rangle$$
weakly in $L^2(\Omega)$. Then by uniqueness of the limit we can establish that
$$Y_t(\varphi) = \langle X_t^{\cdot}, \varphi\rangle$$
in $L^2(\Omega)$.

Note that there exists a subsequence $n(j)$ such that $\langle X_t^{n(j),\cdot}, \varphi\rangle$ converges for every $\varphi$, that is, $n(j)$ is independent of $\varphi$.

We have that $X_t^{n,\cdot}$ is bounded in the Sobolev norm $ L^2(\Omega, W^{k,p}(U))$ for each $n\geq 0$. Indeed, by Proposition \ref{VI_derivative} we have
\begin{align*}
\sup_{n\geq 0} \|X_t^{n,\cdot}\|_{L^2(\Omega, W^{k,p}(U))}^2 =& \sup_{n\geq 0}\sum_{i=0}^{k} \textrm{E}\left[\|\frac{\partial^i}{\partial x^i}X_t^{n,\cdot}\|_{L^{p}(U)}^2\right]\\
\leq& \sum_{i=0}^{k}\left(  \int_{U} \sup_{n\geq 0}\textrm{E}\left[ \|\frac{\partial^i}{\partial x^i}X_t^{n,x}\|^{p} \right]dx\right) ^{\frac{2}{p}%
}
<&\infty
\end{align*}
for a small enough $H<1/2$.

Since $L^2(\Omega, W^{k,p}(U))$, $p\in(1,\infty)$ is reflexive we get that the set $\{X_t^{n,x}\}_{n\geq 0}$ is weakly compact in $L^2(\Omega, W^{k,p}(U))$. Thus, there exists a subsequence $n(j)$, $j\geq 0$ such that
$$X_t^{n(j),\cdot}  \xrightarrow[j\to \infty]{w} Y \in L^2(\Omega, W^{k,p}(U)).$$

On the other hand, we have proven that $X_t^{n,x} \to X_t^x$ strongly in $L^2(\Omega)$, so by uniqueness of the limit we can conclude that
$$X_t^{\cdot} =Y \in L^2(\Omega, W^{k,p}(U)), \ \ P-a.s.$$

Moreover, we have for all $A\in \mathcal{F},$ $\phi \in C_{c}^{\infty }(U;%
\mathbb{R}^{d}),$ $\alpha =(\alpha ^{(1)},...,\alpha ^{(d)})\in (\mathbb{N}%
_{0})^{d}$ with $\left\vert \alpha \right\vert =\alpha ^{(1)}+...+\alpha
^{(d)}\leq k$ that%
\begin{eqnarray*}
&&E\left[ 1_{A}\left\langle X_{t}^{n_{j},\cdot },D^{\alpha }\phi
\right\rangle \right]  \\
&=&(-1)^{\left\vert \alpha \right\vert }E\left[ 1_{A}\left\langle D^{\alpha
}X_{t}^{n_{j},\cdot },\phi \right\rangle \right]  \\
&&\underset{j\longrightarrow \infty }{\longrightarrow }(-1)^{\left\vert
\alpha \right\vert }E\left[ 1_{A}\left\langle D^{\alpha }Y,\phi
\right\rangle \right] 
\end{eqnarray*}%
and thus%
\begin{equation*}
\left\langle X_{t}^{\cdot },D^{\alpha }\phi \right\rangle =(-1)^{\left\vert
\alpha \right\vert }\left\langle D^{\alpha }Y,\phi \right\rangle \text{, }P-%
\text{a.s.}
\end{equation*}

\end{proof}

\appendix

\section{Technical results}
The following result which is due to \cite[Theorem 1] {DPMN92} provides a compactness criterion for subsets of $L^{2}(\Omega)$ using Malliavin calculus.

\begin{thm}
\label{VI_MCompactness}Let $\left\{ \left( \Omega ,\mathcal{A},P\right)
;H\right\} $ be a Gaussian probability space, that is $\left( \Omega ,%
\mathcal{A},P\right) $ is a probability space and $H$ a separable closed
subspace of Gaussian random variables of $L^{2}(\Omega )$, which generate
the $\sigma $-field $\mathcal{A}$. Denote by $\mathbf{D}$ the derivative
operator acting on elementary smooth random variables in the sense that%
\begin{equation*}
\mathbf{D}(f(h_{1},\ldots,h_{n}))=\sum_{i=1}^{n}\partial
_{i}f(h_{1},\ldots,h_{n})h_{i},\text{ }h_{i}\in H,f\in C_{b}^{\infty }(\mathbb{R%
}^{n}).
\end{equation*}%
Further let $\mathbb{D}^{1,2}$ be the closure of the family of elementary
smooth random variables with respect to the norm%
\begin{align*}
\left\Vert F\right\Vert _{1,2}:=\left\Vert F\right\Vert _{L^{2}(\Omega
)}+\left\Vert \mathbf{D}F\right\Vert _{L^{2}(\Omega ;H)}.
\end{align*}%
Assume that $C$ is a self-adjoint compact operator on $H$ with dense image.
Then for any $c>0$ the set
\begin{equation*}
\mathcal{G}=\left\{ G\in \mathbb{D}^{1,2}:\left\Vert G\right\Vert
_{L^{2}(\Omega )}+\left\Vert C^{-1} \mathbf{D} \,G\right\Vert _{L^{2}(\Omega ;H)}\leq
c\right\}
\end{equation*}%
is relatively compact in $L^{2}(\Omega )$.
\end{thm}

In order to formulate compactness criteria useful for our purposes, we need the following technical result which also can be found in \cite{DPMN92}.

\begin{lem}
\label{VI_DaPMN} Let $v_{s},s\geq 0$ be the Haar basis of $L^{2}([0,T])$. For
any $0<\alpha <1/2$ define the operator $A_{\alpha }$ on $L^{2}([0,T])$ by%
\begin{equation*}
A_{\alpha }v_{s}=2^{k\alpha }v_{s}\text{, if }s=2^{k}+j\text{ }
\end{equation*}%
for $k\geq 0,0\leq j\leq 2^{k}$ and%
\begin{equation*}
A_{\alpha }1=1.
\end{equation*}%
Then for all $\beta $ with $\alpha <\beta <(1/2),$ there exists a constant $%
c_{1}$ such that%
\begin{equation*}
\left\Vert A_{\alpha }f\right\Vert \leq c_{1}\left\{ \left\Vert f\right\Vert
_{L^{2}([0,T])}+\left(\int_{0}^{T}\int_{0}^{T}\frac{\left|
f(t)-f(t^{\prime })\right|^2}{\left\vert t-t^{\prime }\right\vert
^{1+2\beta }}dt\,dt^{\prime }\right)^{1/2}\right\} .
\end{equation*}
\end{lem}

A direct consequence of Theorem \ref{VI_MCompactness} and Lemma \ref{VI_DaPMN} is now the following compactness criteria. See \cite{DPMN92} for a proof.

\begin{cor} \label{VI_compactcrit}
Let a sequence of $\mathcal{F}_T$-measurable random variables $X_n\in\mathbb{D}^{1,2}$, $n=1,2...$, be such that there exists a constant $C>0$ with
$$
\sup_n \textrm{\emph{E}}[|X_n|^2] \leq C ,
$$
$$
\sup_n \textrm{\emph{E}}\left[ \| D_t X_n \|_{L^2([0,T])}^2 \right] \leq C \,
$$
and there exists a $\beta \in (0,1/2)$ such that
$$
\sup_n \int_0^T \int_0^T \frac{\textrm{\emph{E}}\left[ \| D_t X_n - D_{t'} X_n \|^2 \right]}{|t-t'|^{1+2\beta}} dtdt' <\infty
$$
where $\|\cdot\|$ denotes any matrix norm.

Then the sequence $X_n$, $n=1,2...$, is relatively compact in $L^{2}(\Omega )$.
\end{cor}

For the use of the above result we will need to exploit the following technical results.

\begin{lem}\label{VI_doubleint}
Let $H \in (0,1/2)$ and $t\in [0,T]$ be fixed. Then, there exists a $\beta \in (0,1/2)$ such that
\begin{align}\label{VI_intI}
\int_0^t \int_0^t \frac{|K_H(t,\theta') - K_H(t,\theta)|^2}{|\theta'-\theta|^{1+2\beta}}d\theta d\theta ' < \infty.
\end{align}
\end{lem}
\begin{proof}
Let $\theta,\theta'\in [0,t]$, $\theta'<\theta$ be fixed. Write
$$K_H (t,\theta) - K_H(t,\theta') = c_H\left[f_t(\theta) - f_t(\theta') + \left(\frac{1}{2}-H\right) \left(g_t(\theta) - g_t(\theta')\right)\right],$$
where $f_t (\theta):= \left(\frac{t}{\theta} \right)^{H-\frac{1}{2}} (t-\theta)^{H-\frac{1}{2}}$ and $g_t(\theta) := \int_\theta^t \frac{f_u (\theta)}{u}du$, $\theta\in [0,t]$.

We will proceed to estimating $K_H (t,\theta) - K_H(t,\theta')$. First, observe the following fact,
$$\frac{y^{-\alpha} -x^{-\alpha}}{(x-y)^{\gamma}} \leq C y^{-\alpha-\gamma}$$
for every $0<y<x<\infty$ and $\alpha :=(\frac{1}{2}-H) \in (0,1/2)$ and $\gamma < \frac{1}{2}-\alpha$. This implies
\begin{align*}
f_t(\theta) - f_t(\theta') &= \left(\frac{t}{\theta}  (t-\theta)\right)^{H-\frac{1}{2}}-\left(\frac{t}{\theta'} (t-\theta')\right)^{H-\frac{1}{2}}\\
&\leq C \left(\frac{t}{\theta}(t-\theta )\right)^{H-\frac{1}{2} -\gamma }t^{2\gamma }\frac{(\theta -\theta')^{\gamma }}{(\theta \theta')^{\gamma }}\\
&\leq C\frac{(\theta -\theta')^{\gamma }}{(\theta \theta')^{\gamma }}(t-\theta )^{H-\frac{1}{2}-\gamma } \\
&\leq C\frac{(\theta -\theta')^{\gamma }}{(\theta \theta
')^{\gamma }}\theta^{H-\frac{1}{2}-\gamma }(t-\theta )^{H-\frac{1}{2}-\gamma }.
\end{align*}

Further, 
\begin{align*}
g_{t}(\theta )-g_{t}(\theta') &= \int_{\theta }^{t}\frac{f_{u}(\theta )-f_{u}(\theta')}{u}du-\int_{\theta'}^{\theta }\frac{f_{u}(\theta')}{u}du \\
&\leq \int_{\theta }^{t}\frac{f_{u}(\theta )-f_{u}(\theta')}{u}du \\
&\leq C\frac{(\theta -\theta')^{\gamma }}{(\theta \theta')^{\gamma }}\int_{\theta }^{t}\frac{(u-\theta )^{H-\frac{1}{2}-\gamma }}{u}du \\
&\leq C\frac{(\theta -\theta')^{\gamma }}{(\theta \theta')^{\gamma }}\theta^{H-\frac{1}{2}-\gamma }\int_{1}^{\infty }\frac{(u-1)^{H-\frac{1}{2}-\gamma }}{u}du \\
&\leq C\frac{(\theta -\theta')^{\gamma }}{(\theta \theta')^{\gamma }}\theta^{H-\frac{1}{2}-\gamma } \\
&\leq C\frac{(\theta -\theta')^{\gamma }}{(\theta \theta')^{\gamma }}\theta^{H-\frac{1}{2}-\gamma }(t-\theta )^{H-\frac{1}{2}-\gamma }.
\end{align*}

As a result, we have for every $\gamma\in (0,H)$, $0<\theta'<\theta<t<T$,
\begin{align*}
(K_{H}(t,\theta )-K_{H}(t,\theta'))^{2}\leq C_{H,T}\frac{(\theta
-\theta')^{2\gamma }}{(\theta \theta')^{2\gamma }}\theta^{2H-1-2\gamma }(t-\theta )^{2H-1-2\gamma },
\end{align*}
for some constant $C_{H,T}>0$ depending only on $H$ and $T$.

Thus
\begin{align*}
\int_{0}^{t}\int_{0}^{\theta }&\frac{(K_{H}(t,\theta )-K_{H}(t,\theta'))^{2}}{| \theta -\theta'|^{1+2\beta }}d\theta'd\theta  \\
&\leq C\int_{0}^{t}\int_{0}^{\theta }\frac{|\theta -\theta'|^{-1-2\beta +2\gamma }}{(\theta \theta')^{2\gamma }}\theta^{2H-1-2\gamma }(t-\theta )^{2H-1-2\gamma }d\theta'd\theta  \\
& =C\int_{0}^{t}\theta^{2H-1-4\gamma }(t-\theta )^{2H-1-2\gamma
}\int_{0}^{\theta }|\theta -\theta'|^{-1-2\beta +2\gamma }(\theta')^{-2\gamma }d\theta'd\theta  \\
&= C\int_{0}^{t}\theta^{2H-1-4\gamma }(t-\theta )^{2H-1-2\gamma }\frac{\Gamma (-2\beta +2\gamma )\Gamma (-2\gamma +1)}{\Gamma (-2\beta +1)}\theta
^{-2\beta }d\theta  \\
&\leq C\int_{0}^{t}\theta^{2H-1-4\gamma -2\beta }(t-\theta
)^{2H-1-2\gamma }d\theta  \\
&=C\frac{\Gamma (2H-2\gamma )\Gamma (2H-4\gamma -2\beta )}{\Gamma
(4H-6\gamma -2\beta )}t^{4H-6\gamma -2\beta -1}<\infty,
\end{align*}%
for appropriately chosen small $\gamma $ and $\beta$.

On the other hand, we have that
\begin{align*}
\int_{0}^{t}\int_{\theta }^{t}&\frac{(K_{H}(t,\theta )-K_{H}(t,\theta'))^{2}}{|\theta -\theta'|^{1+2\beta }}d\theta' d\theta  \\
&\leq C\int_{0}^{t}\theta^{2H-1-4\gamma }(t-\theta )^{2H-1-2\gamma
}\int_{\theta }^{t}\frac{|\theta -\theta'|^{-1-2\beta +2\gamma }}{(\theta')^{2\gamma }}d\theta'd\theta  \\
&\leq C\int_{0}^{t}\theta^{2H-1-6\gamma }(t-\theta )^{2H-1-2\gamma
} \int_{\theta}^t |\theta -\theta'|^{-1-2\beta +2\gamma } d\theta' d\theta  \\
&=C\int_{0}^{t}\theta^{2H-1-6\gamma }(t-\theta )^{2H-1 -2\beta
}d\theta  \\
&\leq Ct^{4H-6\gamma -2\beta -1}.
\end{align*}%
Hence
\begin{align*}
\int_{0}^{t}\int_{0}^{t}\frac{(K_{H}(t,\theta )-K_{H}(t,\theta'))^{2}}{|\theta -\theta'|^{1+2\beta }}%
d\theta' d\theta <\infty .
\end{align*}
\end{proof}

\begin{lem}\label{VI_iterativeInt}
Let $H \in (0,1/2)$, $\theta,t\in [0,T]$, $\theta<t$ and $(\varepsilon_1,\dots, \varepsilon_{m})\in \{0,1\}^{m}$ be fixed. Assume $w_j+\left(H-\frac{1}{2}-\gamma\right) \varepsilon_j>-1$ for all $j=1,\dots,m$. Then exists a finite constant $C=C(H,T)>0$ such that
\begin{align*}
\int_{\Delta_{\theta,t}^{m}}  &\prod_{j=1}^{m} (K_H(s_j,\theta) - K_H(s_j,\theta'))^{\varepsilon_j} |s_j-s_{j-1}|^{w_j} ds\\
\leq& C^m \left(\frac{\theta-\theta'}{\theta \theta'}\right)^{\gamma \sum_{j=1}^m \varepsilon_j} \theta^{\left( H-\frac{1}{2}-\gamma\right)\sum_{j=1}^m \varepsilon_j} \,  \Pi_{\gamma}(m) \, (t-\theta)^{\sum_{j=1}^m w_j + \left( H-\frac{1}{2}-\gamma\right) \sum_{j=1}^m \varepsilon_j +m}
\end{align*}
for $\gamma \in (0,H)$, where
\begin{align}\label{VI_Pi}
\Pi_{\gamma}(m):=\prod_{j=1}^{m-1} \frac{\Gamma \left(\sum_{l=1}^{j} w_l + \left(H-\frac{1}{2}-\gamma \right)\sum_{l=1}^{j} \varepsilon_l + j\right)\Gamma \left( w_{j+1}+1\right)}{\Gamma \left( \sum_{l=1}^{j+1} w_l + \left(H-\frac{1}{2}-\gamma \right)\sum_{l=1}^{j} \varepsilon_l + j+1\right)}.
\end{align}
Observe that if $\varepsilon_j=0$ for all $j=1,\dots, m$ we obtain the classical formula.
\end{lem}
\begin{rem}\label{VI_remPi}
Observe that
\begin{align*}
\Pi_{\gamma}(m)&\leq \frac{\prod_{j=1}^m\Gamma (w_j +1)}{\Gamma \left(\sum_{j=1}^{m} w_j + \left(H-\frac{1}{2}-\gamma \right)\sum_{j=1}^{m-1} \varepsilon_j + m \right)}\\
&\leq \frac{\prod_{j=1}^m\Gamma (w_j +1)}{\Gamma \left(\sum_{j=1}^{m} w_j + \left(H-\frac{1}{2}-\gamma \right)\sum_{j=1}^{m} \varepsilon_j + m \right)},
\end{align*}
since the function $\Gamma$ is increasing on $(1,\infty)$.
\end{rem}
\begin{proof}
First, we recall the following well-known formula: for given exponents $a,b>-1$ and some fixed $s_{j+1}>s_j$ we have
$$\int_{\theta}^{s_{j+1}} (s_{j+1}-s_j)^{a} (s_j-\theta)^b ds_j =\frac{\Gamma \left( a+1\right)\Gamma \left( b+1\right)}{\Gamma \left( a+b+2\right)} (s_{j+1}-\theta)^{a+b+1}.$$
We recall from Lemma \ref{VI_intI} that for every $\gamma\in (0,H)$, $0<\theta'<\theta<s_j<T$,
\begin{align*}
K_{H}(s_j,\theta )-K_{H}(s_j,\theta')\leq C_{H,T}\frac{(\theta
-\theta')^{\gamma }}{(\theta \theta')^{\gamma }}\theta^{H-\frac{1}{2}-\gamma }(s_j-\theta )^{H-\frac{1}{2}-\gamma },
\end{align*}
for some constant $C_{H,T}>0$ depending only on $H$ and $T$. In view of the above arguments we have
\begin{align*}
\int_{\theta}^{s_2} &|K_H(s_1,\theta)-K_H(s_1,\theta')|^{\varepsilon_1} |s_2-s_1|^{w_2}|s_1-\theta|^{w_1}ds_1\\
&\leq C_{H,T}^{\varepsilon_1} \frac{(\theta
-\theta')^{\gamma\varepsilon_1 }}{(\theta \theta')^{\gamma \varepsilon_1}}\theta^{\left(H-\frac{1}{2}-\gamma\right)\varepsilon_1 }\int_{\theta}^{s_2}|s_2-s_1|^{w_2}|s_1-\theta|^{w_1+\left(H-\frac{1}{2}-\gamma\right)\varepsilon_1}ds_1\\
&= C_{H,T}^{\varepsilon_1} \frac{(\theta
-\theta')^{\gamma\varepsilon_1 }}{(\theta \theta')^{\gamma \varepsilon_1}}\theta^{\left(H-\frac{1}{2}-\gamma\right)\varepsilon_1 } \frac{\Gamma\left(\hat{w}_1\right)\Gamma\left(\hat{w}_2\right)}{\Gamma\left(\hat{w}_1+\hat{w}_2\right)}(s_2-\theta)^{w_1+w_2+\left(H-\frac{1}{2}-\gamma\right)\varepsilon_1+1},
\end{align*}
where
$$\hat{w}_1 := w_1+\left(H-\frac{1}{2}-\gamma\right)\varepsilon_1+1, \quad \hat{w}_2:=w_2+1.$$
Integrating iteratively we obtain the desired formula.
\end{proof}

Finally, we give a similar estimate which is used in Lemma \ref{VI_relcomp}.

\begin{lem}\label{VI_iterativeInt2}
Let $H \in (0,1/2)$, $\theta,t\in [0,T]$, $\theta<t$ and $(\varepsilon_1,\dots, \varepsilon_{m})\in \{0,1\}^{m}$ be fixed. Assume $w_j+\left(H-\frac{1}{2}\right) \varepsilon_j>-1$ for all $j=1,\dots,m$. Then exists a finite constant $C>0$ such that
\begin{align*}
\int_{\Delta_{\theta,t}^{m}}  &\prod_{j=1}^{m} (K_H(s_j,\theta))^{\varepsilon_j} |s_j-s_{j-1}|^{w_j} ds\\
&\leq C^m \theta^{\left( H-\frac{1}{2}\right)\sum_{j=1}^m \varepsilon_j} \,  \Pi_0(m) \, (t-\theta)^{\sum_{j=1}^m w_j + \left( H-\frac{1}{2}\right) \sum_{j=1}^m \varepsilon_j +m}
\end{align*}
for $\gamma \in (0,H)$, where $\Pi_0$ is given as in \eqref{VI_Pi}. Observe that if $\varepsilon_j=0$ for all $j=1,\dots, m$ we obtain the classical formula.
\end{lem}
\begin{rem}\label{VI_remPi2}
Observe that
$$\Pi_0(m)\leq \frac{\prod_{j=1}^m\Gamma (w_j +1)}{\Gamma \left(\sum_{j=1}^{m} w_j + \left(H-\frac{1}{2} \right)\sum_{j=1}^m \varepsilon_j + m\right)},$$
due to the fact that $\Gamma$ is increasing on $(1,\infty)$.
\end{rem}
\begin{proof}
By similar arguments as in the proof of Lemma \ref{VI_intI} it is easy to derive the following estimate
$$|K_H(s_j,\theta)| \leq C_{H,T} |s_j-\theta|^{H-\frac{1}{2}}\theta^{H-\frac{1}{2}}$$
for every $0<\theta<s_j<T$ and some constant $C_{H,T}>0$. This implies
\begin{align*}
&\int_{\theta}^{s_2} (K_H(s_1,\theta))^{\varepsilon_1} |s_2-s_1|^{w_2}|s_1-\theta|^{w_1}ds_1\\
&\leq C_{H,T}^{\varepsilon_1} \, \theta^{\left(H-\frac{1}{2}\right)\varepsilon_1} \int_{\theta}^{s_2} |s_2-s_1|^{w_2} |s_1-\theta|^{w_1+\left(H-\frac{1}{2}\right)\varepsilon_1}ds_1\\
&= C_{H,T}^{\varepsilon_1} \, \theta^{\left(H-\frac{1}{2}\right)\varepsilon_1} \frac{\Gamma\left(w_1+w_2+\left(H-\frac{1}{2}\right)\varepsilon_1+1\right)\Gamma\left(w_2+1\right)}{\Gamma\left(w_1+w_2+\left(H-\frac{1}{2}\right)\varepsilon_1+2\right)}(s_2-\theta)^{w_1+w_2+\left(H-\frac{1}{2}\right)\varepsilon_1+1}
\end{align*}
Integrating iteratively one obtains the desired estimate.
\end{proof}
The next auxiliary result can be found in \cite{LiWei}. 

\begin{lem}\label{LiWei}
Assume that $X_{1},...,X_{n}$ are real centered jointly Gaussian
random variables, and $\Sigma =(E[X_{j}X_{k}])_{1\leq j,k\leq n}$ is the
covariance matrix, then%
\begin{equation*}
E[\left\vert X_{1}\right\vert ...\left\vert X_{n}\right\vert ]\leq \sqrt{%
perm(\Sigma )},
\end{equation*}%
where $perm(A)$ is the permanent of a matrix $A=(a_{ij})_{1\leq i,j\leq n}$
defined by%
\begin{equation*}
perm(A)=\sum_{\pi \in S_{n}}\prod_{j=1}^{n}a _{j,\pi (j)}
\end{equation*}%
for the symmetric group $S_{n}$.

\end{lem}

The next result corresponds to Lemma 3.19 in \cite{CD}:
\begin{lem}\label{CD}
Let $Z_{1},...,Z_{n}$ be mean zero Gaussian variables which are linearly
independent. Then for any measurable function $g:\mathbb{R}\longrightarrow 
\mathbb{R}_{+}$ we have that%
\begin{equation*}
\int_{\mathbb{R}^{n}}g(v_{1})\exp (-\frac{1}{2}Var[%
\sum_{j=1}^{n}v_{j}Z_{j}])dv_{1}...dv_{n}=\frac{(2\pi )^{(n-1)/2}}{(\det
Cov(Z_{1},...,Z_{n}))^{1/2}}\int_{\mathbb{R}}g(\frac{v}{\sigma _{1}})\exp (-%
\frac{1}{2}v^{2})dv,
\end{equation*}%
where $\sigma _{1}^{2}:=Var[Z_{1}\left\vert Z_{2},...,Z_{n}\right] $.

\end{lem}

\begin{lem}\label{OrderDerivatives}
Let $n,$ $p$ and $k$ be non-negative integers, $%
k\leq n$. Assume we have functions $f_{j}:[0,T]\rightarrow \mathbb{R}$, $%
j=1,\dots ,n$ and $g_{i}:[0,T]\rightarrow \mathbb{R}$, $i=1,\dots ,p$ such
that 
\begin{equation*}
f_{j}\in \left\{ \frac{\partial ^{\alpha _{j}^{(1)}+...+\alpha _{j}^{(d)}}}{%
\partial ^{\alpha _{j}^{(1)}}x_{1}...\partial ^{\alpha _{j}^{(d)}}x_{d}}%
b^{(r)}(u,X_{u}^{x}),\text{ }r=1,...,d\right\} ,\text{ }j=1,...,n
\end{equation*}%
and 
\begin{equation*}
g_{i}\in \left\{ \frac{\partial ^{\beta _{i}^{(1)}+...+\beta _{i}^{(d)}}}{%
\partial ^{\beta _{i}^{(1)}}x_{1}...\partial ^{\beta _{i}^{(d)}}x_{d}}%
b^{(r)}(u,X_{u}^{x}),\text{ }r=1,...,d\right\} ,\text{ }i=1,...,p
\end{equation*}%
for $\alpha :=(\alpha _{j}^{(l)})\in \mathbb{N}_{0}^{d\times n}$ and $\beta
:=(\beta _{i}^{(l)})\in \mathbb{N}_{0}^{d\times p},$ where $X_{\cdot }^{x}$
is the strong solution to 
\begin{equation*}
X_{t}^{x}=x+\int_{0}^{t}b(u,X_{u}^{x})du+B_{t}^{H},\text{ }0\leq t\leq T
\end{equation*}%
for $b=(b^{(1)},...,b^{(d)})$ with $b^{(r)}\in \mathcal{S}(\mathbb{R}^{d})$
for all $r=1,...,d$. So (as we shall say in the sequel) the product $%
g_{1}(r_{1})\cdot \dots \cdot g_{p}(r_{p})$ has a total order of derivatives 
$\left\vert \beta \right\vert =\sum_{l=1}^{d}\sum_{i=1}^{p}\beta _{i}^{(l)}$%
. We know from Lemma \ref{partialshuffle} that 
\begin{align}
& \int_{\Delta _{\theta ,t}^{n}}f_{1}(s_{1})\dots f_{k}(s_{k})\int_{\Delta
_{\theta ,s_{k}}^{p}}g_{1}(r_{1})\dots g_{p}(r_{p})dr_{p}\dots
dr_{1}f_{k+1}(s_{k+1})\dots f_{n}(s_{n})ds_{n}\dots ds_{1}  \notag \\
& =\sum_{\sigma \in A_{n,p}}\int_{\Delta _{\theta ,t}^{n+p}}h_{1}^{\sigma
}(w_{1})\dots h_{n+p}^{\sigma }(w_{n+p})dw_{n+p}\dots dw_{1},  \label{h}
\end{align}%
where $h_{l}^{\sigma }\in \{f_{j},g_{i}:1\leq j\leq n,$ $1\leq i\leq p\}$, $%
A_{n,p}$ is a subset of permutations of $\{1,\dots ,n+p\}$ such that $%
\#A_{n,p}\leq C^{n+p}$ for an appropriate constant $C\geq 1$, and $%
s_{0}=\theta $. Then the products%
\begin{equation*}
h_{1}^{\sigma }(w_{1})\cdot \dots \cdot h_{n+p}^{\sigma }(w_{n+p})
\end{equation*}%
have a total order of derivatives given by $\left\vert \alpha \right\vert
+\left\vert \beta \right\vert .$
\end{lem}

\begin{proof}
The result is proved by induction on $n$. For $n=1$ and $k=0$ the result is
trivial. For $k=1$ we have 
\begin{eqnarray*}
\int_{\theta }^{t}f_{1}(s_{1})\int_{\Delta _{\theta
,s_{1}}^{p}}g_{1}(r_{1})\dots g_{p}(r_{p}) &&dr_{p}\dots dr_{1}ds_{1} \\
&=&\int_{\Delta _{\theta ,t}^{p+1}}f_{1}(w_{1})g_{1}(w_{2})\dots
g_{p}(w_{p+1})dw_{p+1}\dots dw_{1},
\end{eqnarray*}%
where we have put $w_{1}=s_{1},$ $w_{2}=r_{1},\dots ,w_{p+1}=r_{p}$. Hence
the total order of derivatives involved in the product of the last integral
is given by $\sum_{l=1}^{d}\alpha
_{1}^{(l)}+\sum_{l=1}^{d}\sum_{i=1}^{p}\beta _{i}^{(l)}=\left\vert \alpha
\right\vert +\left\vert \beta \right\vert .$

Assume the result holds for $n$ and let us show that this implies that the
result is true for $n+1$. Either $k=0,1$ or $2\leq k\leq n+1$. For $k=0$ the
result is trivial. For $k=1$ we have 
\begin{align*}
\int_{\Delta _{\theta ,t}^{n+1}}& f_{1}(s_{1})\int_{\Delta _{\theta
,s_{1}}^{p}}g_{1}(r_{1})\dots g_{p}(r_{p})dr_{p}\dots
dr_{1}f_{2}(s_{2})\dots f_{n+1}(s_{n+1})ds_{n+1}\dots ds_{1} \\
& =\int_{\theta }^{t}f_{1}(s_{1})\left( \int_{\Delta _{\theta
,s_{1}}^{n}}\int_{\Delta _{\theta ,s_{1}}^{p}}g_{1}(r_{1})\dots
g_{p}(r_{p})dr_{p}\dots dr_{1}f_{2}(s_{2})\dots
f_{n+1}(s_{n+1})ds_{n+1}\dots ds_{2}\right) ds_{1}.
\end{align*}%
From (\ref{shuffleIntegral}) we observe by using the shuffle permutations
that the latter inner double integral on diagonals can be written as a sum
of integrals on diagonals of length $p+n$ with products having a total order
of derivatives given by $\sum_{l=1}\sum_{j=2}^{n+1}\alpha
_{j}^{(l)}+\sum_{l=1}^{d}\sum_{i=1}^{p}\beta _{i}^{(l)}$. Hence we obtain a
sum of products, whose total order of derivatives is $\sum_{l=1}^{d}%
\sum_{j=2}^{n+1}\alpha _{j}^{(l)}+\sum_{l=1}^{d}\sum_{i=1}^{p}\beta
_{i}^{(l)}+\sum_{l=1}^{d}\alpha _{1}^{(l)}=\left\vert \alpha \right\vert
+\left\vert \beta \right\vert .$

For $k\geq 2$ we have (in connection with Lemma \ref{partialshuffle}) from
the induction hypothesis that 
\begin{align*}
\int_{\Delta _{\theta ,t}^{n+1}}f_{1}(s_{1})\dots f_{k}(s_{k})\int_{\Delta
_{\theta ,s_{k}}^{p}}g_{1}(r_{1})\dots g_{p}(r_{p})& dr_{p}\dots
dr_{1}f_{k+1}(s_{k+1})\dots f_{n+1}(s_{n+1})ds_{n+1}\dots ds_{1} \\
=\int_{\theta }^{t}f_{1}(s_{1})\int_{\Delta _{\theta
,s_{1}}^{n}}f_{2}(s_{2})\dots f_{k}(s_{k})& \int_{\Delta _{\theta
,s_{k}}^{p}}g_{1}(r_{1})\dots g_{p}(r_{p})dr_{p}\dots dr_{1} \\
& \times f_{k+1}(s_{k+1})\dots f_{n+1}(s_{n+1})ds_{n+1}\dots ds_{2}ds_{1} \\
=\sum_{\sigma \in A_{n,p}}\int_{\theta }^{t}f_{1}(s_{1})\int_{\Delta
_{\theta ,s_{1}}^{n+p}}& h_{1}^{\sigma }(w_{1})\dots h_{n+p}^{\sigma
}(w_{n+p})dw_{n+p}\dots dw_{1}ds_{1},
\end{align*}%
where each of the products $h_{1}^{\sigma }(w_{1})\cdot \dots \cdot
h_{n+p}^{\sigma }(w_{n+p})$ have a total order of derivatives given by $%
\sum_{l=1}\sum_{j=2}^{n+1}\alpha
_{j}^{(l)}+\sum_{l=1}^{d}\sum_{i=1}^{p}\beta _{i}^{(l)}.$ Thus we get a sum
with respect to a set of permutations $A_{n+1,p}$ with products having a
total order of derivatives which is%
\begin{equation*}
\sum_{l=1}^{d}\sum_{j=2}^{n+1}\alpha
_{j}^{(l)}+\sum_{l=1}^{d}\sum_{i=1}^{p}\beta _{i}^{(l)}+\sum_{l=1}^{d}\alpha
_{1}^{(l)}=\left\vert \alpha \right\vert +\left\vert \beta \right\vert .
\end{equation*}
\end{proof}

\end{document}